\documentclass[pdflatex,sn-mathphys-num]{sn-jnl}

\usepackage{mathtools} 
\mathtoolsset{showonlyrefs} 

\usepackage{amssymb} 

\usepackage{amsthm} 
\theoremstyle{thmstyleone}%
\newtheorem{theorem}{Theorem}[section]
\newtheorem*{theorem*}{Theorem}

\usepackage{makecell}
\usepackage[dvipsnames]{xcolor} 
\usepackage{microtype} 
\usepackage{lmodern} 
\usepackage[T1]{fontenc} 
\usepackage[inline]{enumitem} 
\setlist[enumerate]{label=(\arabic*)}
\usepackage{booktabs}

\usepackage{zref-clever} 
\zcsetup{cap} 

\usepackage{hyperref}
\usepackage{url}
\hypersetup{
	colorlinks = true,
	linkcolor = NavyBlue,
	citecolor = NavyBlue,
	filecolor = Red,
	urlcolor = Red,
	menucolor = Red,
	runcolor = Red,
	linktoc = all,
	pdftitle = {},
	pdfauthor = {Koustav Mondal},
}

\usepackage[
	openlevel = 2,
	atend,
	numbered,
]{bookmark} 

\numberwithin{equation}{section}
\def \EE#1#2{\mathbb E_{#1,#2}}

\newcommand{\C}{\mathbb C}

\newcommand{\F}{\mathbb F}

\newcommand{\Z}{\mathbb Z}
\newcommand{\N}{\mathbb N}
\newcommand{\Q}{\mathbb Q}

\renewcommand{\H}{\mathbb H}

\newcommand{\E}{\widehat E_2}
\newcommand{\G}{\Gamma}
\renewcommand{\a}{\alpha}
\renewcommand{\b}{\beta}

\renewcommand{\Im}{\operatorname{Im}}
\newcommand{\im}{\operatorname{Im}}
\renewcommand{\pmod}[1]{\ (\text{mod}\ #1)}

\DeclareMathOperator{\SL}{SL}
\DeclareMathOperator{\PSL}{PSL}

\DeclareMathOperator{\lcm}{lcm}

\allowdisplaybreaks[4]

\begin{document}

\AddToHook{env/corollary/begin}{\zcsetup{countertype={theorem=corollary}}} 
\newtheorem{corollary}[theorem]{Corollary}
\newtheorem*{corollary*}{Corollary}

\AddToHook{env/lemma/begin}{\zcsetup{countertype={theorem=lemma}}} 
\newtheorem{lemma}[theorem]{Lemma}
\newtheorem*{lemma*}{Lemma}

\AddToHook{env/sublemma/begin}{\zcsetup{countertype={theorem=sublemma}}} 
\newtheorem{sublemma}{}[theorem]
\newtheorem*{sublemma*}{}

\AddToHook{env/proposition/begin}{\zcsetup{countertype={theorem=proposition}}} 
\newtheorem{proposition}[theorem]{Proposition}
\newtheorem*{proposition*}{Proposition}

\AddToHook{env/conjecture/begin}{\zcsetup{countertype={theorem=conjecture}}} 
\newtheorem{conjecture}[theorem]{Conjecture}
\newtheorem*{conjecture*}{Conjecture}

\theoremstyle{thmstyletwo}%
\AddToHook{env/example/begin}{\zcsetup{countertype={theorem=example}}} 
\newtheorem{example}[theorem]{Example}
\newtheorem*{example*}{Example}
\theoremstyle{remark}
\AddToHook{env/remark/begin}{\zcsetup{countertype={theorem=remark}}} 
\newtheorem{remark}[theorem]{Remark}
\newtheorem*{remark*}{Remark}
\AddToHook{env/notation/begin}{\zcsetup{countertype={theorem=notation}}} 
\newtheorem{notation}[theorem]{Notation}
\newtheorem*{notation*}{Notation}
\newenvironment{newthm}[2][definition]
{\let\tempthm\relax%
	\theoremstyle{#1}%
	\newtheorem*{tempthm}{#2}%
	\tempthm}
{\endtempthm}

\theoremstyle{thmstylethree}%
\theoremstyle{definition}
\AddToHook{env/definition/begin}{\zcsetup{countertype={theorem=definition}}} 
\newtheorem{definition}[theorem]{Definition}
\newtheorem*{definition*}{Definition}

\title[ Relating elliptic curve point-counting and solutions of  quadratic forms with congruence conditions 

]
{
Relating elliptic curve point-counting and solutions of  quadratic forms with congruence conditions
}


\author*[]{\fnm{Koustav} \sur{Mondal}}\email{kmonda1@lsu.edu}



\affil*[]{\orgdiv{Department of Mathematics}, \orgname{Louisiana State University}, \orgaddress{\city{Baton Rouge}, \postcode{70803}, \state{LA}, \country{USA}}}




\abstract{In this paper, we analyze the theta series associated to the quadratic form $Q(\Vec{x}) \coloneqq x_1^2+x_2^2+x_3^2+x_4^2$ with congruence conditions on $x_i$ modulo $2,3,4$ and $6$. 
		By employing special operators on modular, non-holomorphic Eisenstein series of weight 2, we  construct a basis for Eisenstein space for levels $2^k, k\leq 7$, $3^{\ell}, \ell\leq 3$ and $p$, for odd prime $p$. 
		Using the relation between the trace of Frobenius on an elliptic curve and the Fourier coefficients of the cusp form part of theta series  corresponding to $Q$, we establish relation between the number of integer solutions to the equation $Q(\Vec{x}) = p$  and the number of $\F_p$-rational points on the associated elliptic curve under certain congruence conditions on $p$.}




\maketitle

\section{Introduction}

Analyzing representations of natural numbers by quadratic forms has a long history in mathematics.\cite{, Bhargava, Bhargava-Hanke,beli2022universalintegralquadraticforms,book:Cox,Dickson}.  In this paper, we prove that counts of representations by quadratic forms with various congruence conditions are linearly related to counting points on elliptic curves over finite fields. 

We begin by setting the notation and terminology for quadratic forms. A quadratic   form $Q(\vec{x})$, for $\vec{x}\in \Z^n$, is degree 2 homogeneous polynomial   in $n$ variables. We will only focus on quadratic forms with integer coefficients. Let $A\in M_n(\Z)$ be a positive-definite, 
symmetric matrix that satisfies
$A_{ii}\equiv 0\pmod{2}.$
Such matrices are called \emph{even}. 
We  say that $A$ is the associated matrix to the positive-definite quadratic form $Q(\vec{x})$  if 
\begin{equation}
	Q(\vec{x})=\frac{1}{2}\vec{x}^tA\vec{x}.
\end{equation}

A quadratic form $Q$ \emph{represents} a number $\ell\in \Z$ if there is exists some $\Vec{x}_0\in \Z^n$ such that $Q(\Vec{x}_0)=\ell$. Quadratic forms that represent all natural numbers are termed as \emph{universal quadratic forms}. Lagrange proved in 1770 that every positive integer can be expressed as the sum of four squares.  Following that, Ramanujan \cite{Ramanujan} gave a list of all universal quaternary forms . Later,  Bhargava and Hanke \cite{Bhargava, Bhargava-Hanke} proved all the cases for universal quaternary quadratic forms. 

For a positive-definite quadratic form $Q(\Vec{x})$, the associated theta series $\theta_Q$ is defined as
\begin{equation}
	\theta_Q(\tau)\coloneqq\sum_{\substack{\Vec{x}\in  \mathbb{Z}^n}}q^{Q(\Vec{x})}, \hspace{20pt}q\coloneqq e^{2\pi i \tau},  
\end{equation}
for $\tau \in \H \coloneqq \{z \in \C\mid \im(z)>0\}$. 

Employing associated theta series is a conventional approach to analyze representation of numbers by the quadratic form. Recall that modular forms are holomorphic functions on upper-half complex plane $\H$ that satisfy certain transformation properties with respect to the action of a congruence subgroup $\G\subset \SL_2(\Z)$.
Consequently, a modular form $f(\tau)$ has a Fourier expansion 
(see \zcref{sec:background} for details).   
Theta series, which are modular forms, provide insights into the representation of numbers by the associated quadratic forms by virtue of analysis of the coefficients of the Fourier series. 
\begin{definition}
	A \emph{congruent quadratic form} is a pair $(Q(\Vec{x}), (\Vec{a}, \Vec{S}))$ where $Q(\Vec{x})$ is a quadratic form with the congruence conditions on the arguments $x_i \equiv a_i \pmod{S_i}$. The theta series associated with a congruent quadratic form is called the \emph{congruent theta series} 
	\begin{equation}
		\theta_{(Q(\Vec{x}), (\Vec{a}, \Vec{S}))}(\tau)=\sum_{\substack{\Vec{x}\in \Z \\ \Vec{x}\equiv \Vec{a}\hspace{-2.5pt}\pmod{\Vec{S}}}}q^{Q(\Vec{x})}.
	\end{equation}
	
\end{definition}

\begin{example}
    Recall, the classic Jacobi theta functions (following mainly \cite{book:Zagier123}). 
\begin{align}\label{eq:jacobitheta}
	\theta_2(\tau) \coloneqq \sum_{n \in \Z}q^{\frac{(2n+1)^2}{8}}, \quad  \theta_3(\tau)\coloneqq \sum_{n \in \Z}q^{\frac{n^2}{2}},\hspace*{10pt}.
\end{align}
Then, \begin{align}
    \theta_{2}^4(4\tau)=\theta_{\frac{1}{2}(x_1^2+x_2^2+x_3^2+x_4^2),(\vec{1},\vec{2})}, \quad \text{and}\quad  \theta_3(2\tau)^4= \theta_{x_1^2+x_2^2+x_3^2+x_4^2,(\vec{0},\vec{1})}, 
\end{align}
where $\vec{a}=(a,a,a,a)\in \Z^4, a\in \Z$. 

\end{example}

Such congruent theta series were studied in \cite{sun_universalsumspolygonalnumbers, beli2022universalintegralquadraticforms}. 
The congruent theta series, similar to conventional theta series, satisfies modularity properties of certain level while the weight depends on the number of variables (\cite[Chapter IX]{book:Schoeneberg}).  In \cite{CHO2018999}, Cho has shown in some cases coefficients of congruent theta series can be expressed to simple divisor sum functions.

Decomposing the congruent theta series into Eisenstein series and a cusp form reveals  properties about the coefficients of the congruent theta series. In this paper, we will deal with the congruent quadratic form $(Q(\vec{x}), (a_i,S_i))$ where $Q(\Vec{x})=x_1^2+x_2^2+x_3^2+x_4^2$, so we drop the $Q$ in our notations. Also, we will only deal with the column vectors $\Vec{a}=s\cdot\Vec{1}$ and $\Vec{S}=M\cdot\Vec{1}$, so we will use the notation $\theta_{s,M}$. So, 
\begin{equation}
	\theta_{s,M}(\tau)=\sum_{\substack{\Vec{x}\in \Z^4 \\ x_i\equiv s\pmod{M}}}q^{x_1^2+x_2^2+x_3^2+x_4^2}=\sum_{n =1}^{\infty}r_{s,M}(n)q^n, 
\end{equation}
where 
\begin{equation}
	r_{s,M}(n)= \text{ number of elements in  } \{\Vec{x}\in \mathbb{Z}^4 \mid x_1^2+x_2^2+x_3^2+x_4^2=n, \, x_i\equiv s\pmod{M} \}.
\end{equation}

The congruent theta series $\theta_{s,M}$, which is a modular form of weight 2, decomposes into Eisenstein 
and cusp form parts. \cite[Section 2.1]{sun_universalsumspolygonalnumbers}.

We denote the space of modular forms and the Eisenstein series of weight $k$ on $\G$ 
by $\mathcal{M}_k(\G)$ and $\mathcal{E}_k(\G)$ respectively.

Recently, Bringmann and Kane \cite{MR4519810}  demonstrated a technique of applying two operators denoted by $S$ (\hspace{-3pt} Sieving operator) and $V$ to $\widehat{E}_2(\tau)=1-\frac{3}{\pi\Im(\tau)}-24\sum_{n \geq 1}\sigma_1(n)q^n$ to form holomorphic Eisenstein series with modularity of desired levels, where $\displaystyle{\sigma_1(n)=\sum_{d|n}d}$ is the divisor sum function. 
With that intension, we introduce the two operators \cite[Section 2.2]{MR4519810}  
\begin{definition}
	Suppose that $f$ is a function on $\H$ that
	is represented by a Fourier expansion 
	\begin{equation}
		f(\tau)=\sum_{n\geq 0}c_{f}(n)q^n.
	\end{equation}
	We define the \textit{sieving operator} for $M\in \N$ and $m\in \Z$ as
	\begin{equation}
		f|S_{M,m}(\tau)\coloneqq\sum_{\substack{n\geq 0 \\ n\equiv m \hspace{-2pt}\pmod{M}}}c_{f}(n)q^n .
	\end{equation}
	We define the \textit{V operator} for $d \in \Q$ as
	
	\begin{equation}
		f|V_d(\tau)\coloneqq f(d\tau)=\sum_{n\geq 0} c_{f}(n)q^{dn} .
	\end{equation}
	
\end{definition}

These two operators can be used to obtain a basis for weight $2$ Eisenstein spaces of lower levels. Precisely, we show the following theorem.

\begin{theorem}\label{thm:basisgenerationthm}
	
	A basis for the spaces $\mathcal{E}_2(\G_0(p^k))$ can be obtained as follows:
	\begin{enumerate}[label=\textnormal{(\roman*)}]
		\item For any prime $p$ and odd $k$, a basis of $\mathcal{E}_2(\Gamma_0(p^k))$ comprises elements from set
		
		\begin{equation}
			\{f|V_p :  f \in \text{ basis of }( \mathcal{E}_2(\Gamma_0(p^{k-1})))\}.
		\end{equation} 
		\item For $p=2$ and $k\leq 3$, a basis for $\mathcal{E}_2(\Gamma_0(2^{2k}))$  comprises elements from the set 
		\begin{equation}
			\{\widehat{E}_2|S_{2^k,m}(\tau)\mid  m\in (\Z/2^k\Z)^{\times}\}\cup \{f|V_2 :  f \in\text{basis of }(\mathcal{E}_2(\Gamma_0(p^{k-1}))) \}. 
		\end{equation}Applying $S$ and $V$ operators to $\E$ cannot generate a full basis for $\mathcal{E}_2(\G_0(2^k))$ when $k>3$.

		\item For $p=3$, a basis for $\mathcal{E}_2(\Gamma_0(3^2))$ can be given by 
		\begin{equation}\widehat{E}_2|S_{3,1}(\tau),  \quad \widehat{E}_2|S_{3,2}(\tau), \quad \widehat{E}_2|(3V_3-S_{3,0})(\tau).
		\end{equation}
		Applying $S$ and $V$ operators to $\E$ cannot generate a full basis for $\mathcal{E}_2(\G_0(3^{2k}))$ when $k>1$. 
		\item For any other odd prime $p$, it is not possible to obtain a full basis for $ \mathcal{E}_2(\Gamma_0(p^k)), k>1$ using only $S$ and $V$ operators. 
	\end{enumerate}
\end{theorem}

The advantage of using these bases is that they provide the Fourier coefficients of the basis elements along with exact values at the cusps for each Eisenstein element \cite[Section 3.2]{MR4519810}. \zcref{app:evenbasis} and \zcref{app:oddbasis} give a complete list of basis elements using \zcref{thm:basisgenerationthm}. 
The basis of Eisenstein series provides a powerful framework for tackling various mathematical problems, such as the study of integer partitions  (see \cite{ono, Rollen}), quadratic forms, and many other areas. In this paper, as an application of \zcref{thm:basisgenerationthm}, we can extract Eisenstein series component of a given congruent theta series. If we look at the congruent quadratic form $Q_{1,3}$, 
then from the definition, $Q_{1,3}$ 
can only represent a natural number $n$ if $n \equiv 1\pmod{3}$.  For any natural number $n$, let
\begin{equation}
	r_{1,3}(n)= \text{ number of elements in  } \{\Vec{x}\in \mathbb{Z}^4 \mid x_1^2+x_2^2+x_3^2+x_4^2=n, \, x_i\equiv 1\pmod{3} \}.
\end{equation}  

\begin{remark}
	Notice that $r_{1,3}(n)$ is the number of elements of $\{\Vec{x}\in \Z^4 \mid x_1^2+x_2^2+x_3^2+x_4^2=n, x_i\equiv 1\pmod{3}\}$, including all permutations of the solution tuple $(x_1,x_2,x_3,x_4)$. When $n$ is odd, the tuples $(x_1, x_2,x_2,x_4)$ can only have parities such that $x_1+x_2+x_3+x_4$ is odd. Accounting for permutations, each such solution tuple $(x_1,x_2,x_3,x_4)$ contributes $4$ to $r_{1,3}(n)$. As a result, we get the following proposition.

\end{remark}
\begin{proposition}\label{prop:r13isdivby4}
	For any odd $n\in \N$,  $r_{1,3}(n)$ will be divisible by $4$.
\end{proposition}
In order to learn more about $r_{1,3}(n)$, we need to use  modularity of the generating function $\theta_{1,3}(\tau)$. Decomposing the congruent theta series and analyzing the cusp form part of $\eta(6\tau)^4$ and the elliptic curve associated with the cusp form leads us to a relation between them.

\begin{theorem}\label{thm:r13ecpc}
	Let $E$ be the elliptic curve $y^2=x^3+1$ defined over $\Q$ and let $N_p(k)$ denote the number of  $\mathbb{F}_{p^k}$ rational points on $E$ (including the point at infinity). Then, for primes $p\equiv 1\pmod{6}$,
	\begin{equation}
		r_{1,3}(p)=\frac{1}{3}N_p(1).
	\end{equation}
	In general, we have the following three-term relation 
	
	\begin{equation}\label{eq:equationforrk}
		3r_{1,3}(p^k)= 
		\begin{cases}
			N_p(k)+3p\cdot r_{1,3}(p^{k-2}) \hspace{15pt} & \textnormal{ for } k \geq 3 \\
			N_p(k)                                        & \textnormal{ for } k=1,2.
		\end{cases}
	\end{equation}
	Expanding this gives the following relation,
	\begin{equation}\label{eq:rkodd}
		r_{1,3}(p^{2k+1})=\frac{1}{3}[N_p(2k+1)+pN_p(2k-1)+ \cdots +p^kN_p(1)]\hspace*{10pt} \text{ for }p\equiv 1\pmod{6},
	\end{equation}
	and
	\begin{equation}\label{eq:rkeven}
		r_{1,3}(p^{2k+2})=\frac{1}{3}[N_p(2k+2)+pN_p(2k)+ \cdots +p^kN_p(2)]\hspace*{10pt} \text{ for prime }p>3.
	\end{equation}

\end{theorem}

\begin{example}
	For $p=103$, $r_{1,3}(103)=28$, and the Fourier coefficient $a_{103}$ of $q^{103}$ in the Fourier expansion of $\eta(6\tau)^4 $ is  20. We can easily verify that 
	\begin{equation}
		3r_{1,3}(p)= p+1-a_p=\#\{(x,y)\in \F_p\times\F_p\mid \, y^2=x^3+1 \pmod{p} \}.
	\end{equation}

\end{example}

We give a generalization of \zcref{thm:r13ecpc}. We can extend the formula for $r_{1,3}(n)$, where $n\equiv 1\pmod{6}$. In terms, this relation shows a more general relationship between congruent quadratic forms and point counting on elliptic curves.

\begin{theorem}\label{thm:r13general}
	Consider any set of distinct primes $p_1,\dots,p_k$ and natural numbers $\a_1,\dots,\a_k$ such that $n=\prod_{i}p_i^{\a_i}\equiv 1\pmod{6}$ with the conditions 
	\begin{equation}
		\a_i\in \begin{cases}
		2\mathbb{N}, & \textnormal{ if } p_i\equiv \textnormal{5\pmod6}, \\
		\mathbb{N},  & \textnormal{ if } p_i\equiv \textnormal{1\pmod6}.
	\end{cases}
	\end{equation}
	Then we have the following decomposition of $r(n)$:
	\begin{equation}
		r_{1,3}(n)=\sum_{i=1}^{k}\left(\prod_{j=i+1}^{k}\left(\sigma_1(p_j^{\a_j})\right)\prod_{j=1}^{i-1}\left(a_{p_j^{\a_j}}\right)\right)r_{1,3}(p_i^{\a_i}),
	\end{equation}
	where $a_\ell$ is the $\ell$-th Fourier coefficient of the cusp form $\eta(6\tau)^4$.

\end{theorem}

\begin{remark}
	From \zcref{prop:r13isdivby4} and \zcref{thm:r13ecpc}, we get that $N_p(1)$ is divisible by $12$ whenever $p\equiv 1 \pmod{6}$. Also, as a direct consequence of \zcref{thm:r13ecpc}, it is evident that
	\begin{equation}
		3r_{1,3}(p^m)\equiv N_p(m) \pmod{p}
	\end{equation}for any $m \in \N$.
\end{remark}  
\begin{remark}
	In \cite[\S 3]{CHO2018999}, some examples can also give rise to similar results listed in \zcref{tab:lineartable}.
	\begin{table}[!ht]
		\centering
		\scalebox{0.829}{
			\begin{tabular}{c c c c c c}
				\toprule
				\makecell{Quadratic \\ form $Q$} & \makecell{Congruence \\ conditions} & \makecell{Elliptic \\ curve $E$\\} & {CM} & \makecell{Relation \\ between \\ $N_p^E(1)$ \\ and $r_Q(p)$} & \makecell{ For primes \\$p$  with \\congruence \\conditions}\\ \midrule
				\makecell{$x_1^2+x_2^2+2x_3^2+2x_4^2$} & \makecell{ $x_1, x_3, x_4 \equiv 1\pmod{2}$,\\$x_2\equiv 0\pmod{2}$} & $y^2=x^3+4x$ & $\Q(\sqrt{-1})$ & $r_Q(p)=N_p^E(1)$ & $p \equiv 1\pmod{4}$\\[15pt]
				\makecell{$x_1^2+x_1x_2+x_2^2$\\ $+x_3^2+x_3x_4+x_4^2$} & \makecell{ $x_1,x_2, x_3 \equiv 1\pmod{3}$,\\$x_4\equiv 0\pmod{3}$} & $y^2+y=x^3$ & $\Q(\sqrt{-3})$ & $r_Q(p)=N_p^E(1)$ & $p \equiv 1\pmod{3}$\\[15pt]
				\makecell{$x_1^2+x_1x_2+2x_2^2$\\$+x_3^2+x_3x_4+2x_4^2$} & \makecell{ $x_1, x_4 \equiv 1\pmod{2}$,\\$x_2,x_3 \equiv 0\pmod{2}$} & \parbox{1.8 cm}{$y^2+xy+y=x^3-x$ } & None & $r_Q(p)=2N_p^E(1)$ & \makecell{$p \equiv 1\pmod{2},$\\$ p \not \equiv 7\pmod{14}$}\\ \bottomrule
			\end{tabular}}
		\caption{Relations similar to \zcref{thm:r13ecpc}, where $r_Q(p)$ is the number of representations of a prime $p$ by the congruent quadratic form $(Q,(a_i,S_i))$ and $N_p^E(1)$ is $|\{(x,y)\in \F_p\mid (x,y) \text{ is on }E\}|+1$ (including the point at infinity) for the elliptic curve $E$.}
		\label{tab:lineartable}
	\end{table}
	
    The basis of the Eisenstein series generated by elements of the form $\widehat{E}_2|S_{M,m}, \widehat{E}_2|S_{M,m}|V_d$ serves as a crucial tool for extracting the Eisenstein component of the associated theta series in all the cases.

\end{remark}

This paper is organized as follows. In \zcref{sec:background}, we recall the properties of modular forms, congruent theta series and its modularity, and explore some properties of the $S$ and $V$ operators. \zcref{sec:basissection} describes the method to obtain a basis of the Eisenstein series space using these operators. In \zcref{sec:theta2,sec:theta4,sec:theta3,sec:theta6}, we have analyzed the coefficients of all the possible cases of congruent theta series $\theta_{k,M}$ for $M=2,4,3,6$ respectively, for all $k$. Finally, \zcref{sec:sectionforECPC} gives a detailed proof of \zcref{thm:r13ecpc,thm:r13general} and some further consequences of our results.

\backmatter





\bmhead{Acknowledgments}

The author extends heartfelt gratitude to Dr.~Fang-Ting Tu and Dr.~Gene S. Kopp for their unwavering support and guidance throughout the course of this project. Special thanks are also due to Dr.~Jingbo Liu for consistently providing supporting materials as needed. Additionally, the author wishes to acknowledge Dr.~Ling Long, Brian Grove, Dr.~Michael Allen, Dr.~Hasan Saad, Esme Rosen, Paresh Singh Arora, and Murtadha Aljanabi for their valuable feedback and insightful contributions.

The author is partially supported by NSF grants DMS-2302514 and DMS-2302531 and the 2024 summer research assistantship awarded by the Department of Mathematics at Louisiana State University.

\section{Background}\label{sec:background}
In this section, we recall the necessary background for modular forms and congruent theta series. Additionally, we define and mention some necessary properties of the $S$ and $V$ operators.

\subsection{Modular forms}
We will adopt the notations and details from Chapter 3 of \cite{book:koblitz}.
A modular form of weight $k\in \mathbb{Z}$ for a finite index subgroup  $\Gamma\subset \SL_2(\Z) $ is a function
$f : \H \rightarrow \mathbb{C}$ satisfying the following three conditions.  \begin{itemize}
	\item {$f$ is holomorphic on $\H$,}
	\item{$f(\gamma(\tau))=f\left(\frac{a\tau+b}{c\tau+d}\right)=(c\tau+d)^kf(\tau)$ for $\tau\in \H$, where $\gamma=\begin{psmallmatrix}
			a&b\\c&d
		\end{psmallmatrix}\in \Gamma$},
	\item {$f|_k\gamma(\tau)\coloneqq(c\tau+d)^{-k}f(\gamma(\tau) )$ is bounded as $\im(\gamma(\tau) ) \rightarrow \infty, \text{ for all } \gamma\in \SL_2(\Z)$.}
\end{itemize}

We denote the space of modular forms of weight $k$ on $\Gamma$ by $\mathcal{M}_k(\Gamma)$. For modular form $f\in \mathcal{M}_k(\Gamma)$, if $(c\tau+d)^{-k}f(\gamma(\tau) )\rightarrow 0$ as $\im(\gamma(\tau) ) \rightarrow \infty \text{ for all } \gamma\in SL_2(\Z)$, we call $f$ a cusp form. We will use
$\mathcal{S}_k(\Gamma)$ to denote the cuspidal forms subspace of $\mathcal{M}_k(\Gamma)$. 

Next, we recall the  Petersson inner product that will help us to decompose a modular forms into cuspidal parts and non-cuspidal parts to analyze them separately. 
\begin{definition} 
	Let $\Gamma\subset \SL_2 (\Z)$ be a congruence subgroup and, $F$ be a fundamental domain for $\Gamma$. Let $f,g\in \mathcal{M}_k(\Gamma)$ with at least one of $f,g$ is a cusp form. The  \emph{Petersson inner product} of $f$  and $g$ is defined as
	\begin{equation}
		\langle f,g\rangle\coloneqq\frac{1}{[\PSL_2(\Z):\Bar{\Gamma}]}\int_{F}f(z)\overline{g(z)}y^k\frac{dxdy}{y^2},
	\end{equation}
	where
	\begin{equation}
		\Bar{\G}=\begin{cases}
			\G/\{\pm I\}, & \text{if }- I\in \G, \\
			\G,           & \text{ otherwise, }
		\end{cases}
	\end{equation}
	where $I$ is the identity matrix.
\end{definition}

In addition, we denote the Eisenstein space, which is the orthogonal complement of $\mathcal{S}_k(\Gamma)$ with respect to the Petersson inner product in $\mathcal{M}_k(\Gamma)$, by $\mathcal{E}_k(\G)$.   That is, we have 
\begin{equation}\label{eq:modulardecomposition}
	\mathcal{M}_k(\Gamma)=\mathcal{E}_k(\Gamma)\oplus \mathcal{S}_k(\Gamma). 
\end{equation}

Assume that the translation matrix $T=\begin{psmallmatrix}
	1&1\\0&1
\end{psmallmatrix}\in \Gamma$.  If $f\in \mathcal{M}_k(\Gamma)$, then $f(\tau+1)=f(\tau)$ for $\tau\in \H$; that is, $f$ is  translation invariant. Therefore, $f$ admits a Fourier expansion 
taking the form 
\begin{equation}
	f(\tau)=\sum_{n\geq 0}c_{f}(n)q^n, \hspace*{15pt}q=e^{2\pi i \tau}. 
\end{equation}
Additionally, for $f\in \mathcal{S}_k(\Gamma)$, we have $c_{f}(0)=0.$

\begin{notation}
	For a positive integer $N$, we  use the standard notation for the following congruence subgroups:
	\begin{enumerate}
		\item $\Gamma_0(N)\coloneqq\left\{\begin{pmatrix}a&b\\c&d\end{pmatrix}\in \SL_2(\Z) \biggm| c\equiv 0\pmod{N}\right\}$,
		\item $\Gamma_1(N)\coloneqq\left\{\begin{pmatrix}a&b\\c&d\end{pmatrix}\in \Gamma_0(N) \biggm| a\equiv 1\pmod{N}\right\}$. 
	\end{enumerate}
	
\end{notation}
Recall from Section 5.6 in \cite{book:Diamond-Shurman} that, if $f(\tau) \in\mathcal{M}_k(\Gamma_0(N)) $ then $f|V_d(\tau)\in \mathcal{M}_k(\Gamma_0(Nd))$. Also, if $L\mid N$ and $g(\tau)\in \mathcal{M}_k(\Gamma_0(N))$, then $g(\tau)\in \mathcal{M}_k(\Gamma_0(L))$. These forms $f|V_d(\tau), g(\tau)$ are called \emph{oldforms} as they are rising from lower levels.

\subsection{Congruent theta series and quadratic forms}

Next, we discuss a specific kind of modular form, the theta series. This subsection formally defines the congruent theta series and digs deeper into the modularity of it. For more detailed discussions, see Section 2 in \cite{CHO2018999}.  We define the discriminant $D$ of the quadratic form $Q$ with associated matrix $A$ as
\begin{equation}
	D=(-1)^{\frac{n}{2}}\det(A),
\end{equation}
where $n$ is the number of variables in the quadratic form $Q$. 
Define the level $N$ of a quadratic form $Q$ as the smallest natural number $N$ such that $NA^{-1}$ is an even matrix. Later in this section, we will see how the level and the discriminant play a significant role in determining the modularity of the congruent theta series.

\begin{notation}
	From now on, we will denote by $Q_{s,M}({\Vec{x}})\leadsto \theta_{s,M}$ the theta series associated with the quadratic form $Q_{s,M}({\Vec{x}})$ or vice-versa. 
\end{notation} We conclude this subsection with the discussion of modularity of congruent theta series. Recall, the space $\mathcal{M}_k(\Gamma,\chi)$ 
contains modular forms $f$ such that for any $\gamma=\begin{psmallmatrix}
	a&b\\c&d
\end{psmallmatrix}\in \Gamma$, $f$ satisfies
\begin{equation}
	f\!\mid\!_k\gamma (\tau) \coloneqq (c\tau+d)^{-k} f(\gamma \tau)=\chi(d)f(\tau),
\end{equation}
for $\tau \in \H$. 
For the precise modularity of $\theta_{s, M}$, we are going to use Theorems 2.4 and 2.5 from \cite{CHO2018999}. 
\begin{theorem}\label{thm:thetamod}
	Let $Q_{s,M}(\Vec{x})$ is 
	a positive-definite, integral congruent quadratic form of $n$ variables,
	where $n$ is even and $s\in \N$.  
	Let $\theta_{s,M}$ be the associated theta series. 
	Then 
	\begin{equation}
		\theta_{s, M}\in \mathcal{M}_{\frac{n}{2}}\left(\Gamma_0(M^2N)\cap \Gamma_1(M), \left(\frac{D}{.}\right)\right)\cong\bigoplus_{\substack{\chi \pmod{M} \\ \chi(-1)=(-1)^{\frac{n}{2}}}} \mathcal{M}_{\frac{n}{2}}\left(\Gamma_0(M^2N), \chi\cdot\left(\frac{D}{.}\right)\right),
	\end{equation}
	where $D$ and $N$ are the discriminant and level of  $Q_{s,M}({\Vec{x}})$, and $\left(\frac{D}{.}\right)$ is the Kronecker symbol. 
	
\end{theorem}

\begin{remark}
	In our paper, the positive-definite integral congruent quadratic form  is of the form  $Q_{k,M}(\Vec{x}) = x_1^2+x_2^2+x_3^2+x_4^2$, where $x_i\equiv k\pmod{M}$. Hence, the associated matrix of $Q_{k,M}$ is
	\begin{equation}
		A=\begin{pmatrix}
			2 & 0 & 0 & 0 \\
			0 & 2 & 0 & 0 \\
			0 & 0 & 2 & 0 \\
			0 & 0 & 0 & 2
		\end{pmatrix},
	\end{equation}
	which has determinant 16, hence the discriminant is $D = 16$, which implies that $\left(\frac{D}{.}\right)$ is the trivial character, and the level of $A$ is $N=4$. By \zcref{thm:thetamod},
	\begin{equation}
		\theta_{k,M}\in \mathcal{M}_2\left(\Gamma_0(4M^2)\cap \Gamma_1(M)\right)\cong\bigoplus_{\substack{\chi \pmod{M} \\ \chi(-1)=1}}\mathcal{M}_2(\Gamma_0(4M^2), \chi).
	\end{equation}
	
\end{remark}

\subsection{Properties of \texorpdfstring{$S$}{S} and \texorpdfstring{$V$}{V} Operators}

In order to apply the $S$ and $V$ operators, we will need to get a simpler formula for them. Then the following proposition gives us a finite sum formula for the $S$ and $V$ operators. These finite sum formulae will allow us to check holomorphicity of $f|S_{M,m}$ and $f|V_d$ and calculate the values of the Eisenstein elements at different cusps easily. 
\begin{proposition}\label{prop:finsum}
	Let $f$ be a translation invariant function and $m,M,d\in \N$. Then for $\tau \in \H$, 
	\begin{equation}
		f|S_{M,m}(\tau)=\frac{1}{M}\sum_{j=0}^{M-1}f\left(\tau+\frac{j}{M}\right)\zeta_{M}^{-jm}
	\end{equation}
	and 
	\begin{equation}
		f|S_{M,m}|V_d(\tau)=\frac{1}{M}\sum_{j=0}^{M-1}f\left(d\tau+\frac{j}{M}\right)\zeta_{M}^{-jm},
	\end{equation}
	where $\zeta_M = e^{\frac{2\pi i}{M}}$.
\end{proposition}

\begin{proof}
	We start with the Fourier series of $f(\tau)=\sum_{n\geq 0}c_{f,v}(n)q^n$. Then, 
	\begin{equation}
		f\left(\tau+\frac{j}{M}\right)\zeta_M^{-jm}=\sum_{n \geq 0}c_{f}(n)q^n\zeta_M^{ j(n-m)}.
	\end{equation}
	Therefore, taking sum over all residue classes $j $ modulo $M$, we get
	\begin{align}
		\sum_{j=0}^{M-1}\!f\!\left(\tau+\frac{j}{M}\right)\zeta_M^{-jm}&=\sum_{n \geq 0}\sum_{j=0}^{M-1}c_{f}(n)q^n\zeta_M^{ j(n-m)}\\
		&=M\cdot \sum_{n \equiv m\pmod{M}}c_{f}(n)q^n=M \cdot f|S_{M,m}.
	\end{align}
	The second to last equality follows from the fact that
	\begin{equation}
		\sum_{j=0}^{M-1}\zeta_M^{ j(n-m)}=\begin{cases}
			M, & \text{if }n\equiv m\pmod{M}, \\
			0, & \text{otherwise.}
		\end{cases}
	\end{equation}
	
	The second formula comes from the definition $f|S_{M,m}|V_d(\tau)=f|S_{M,m}(d\tau).$
\end{proof}

If $f(\tau)=\widehat{E}_2(\tau) \coloneqq 1-\frac{3}{\pi \im(\tau)}-24\sum_{n \geq 1}\sigma_1(n)q^n$ is the non-holomorphic, weight 2 modular Eisenstein series on $\SL_2(\Z)$, then, using \zcref{prop:finsum}, we can  get the values of $\widehat{E}_2|S_{M,m}$ and $\widehat{E}_2|V_d$ at the cusps. For more details, see section 3 in \cite{MR4519810}. 
\begin{proposition}
	For $\frac{h}{k}\in \Q$, we have
	\begin{equation}\label{eq:ES}
		-\lim_{z\rightarrow 0^+}z^2\widehat{E}_2|S_{M,m}\left(\frac{h}{k}+\frac{iz}{k}\right)=\frac{1}{M^3}\sum_{j=0}^{M-1}\gcd(Mh+jk, Mk)^2\zeta_M^{-jm},
	\end{equation}
	and
	\begin{equation}\label{eq:ESV}
		-\lim_{z\rightarrow 0^+}z^2\widehat{E}_2|S_{M,m}|V_d\left(\frac{h}{k}+\frac{iz}{k}\right)=\frac{1}{M^3d^2}\sum_{j=0}^{M-1}\gcd(Mdh+jk, Mk)^2\zeta_M^{-jm}.
	\end{equation}
	
\end{proposition}
\begin{remark}
	If we put $M=1 $ and  $m=0$ into \zcref{eq:ESV}, we get
	\begin{equation}\label{eq:EV}
		-\lim_{z\rightarrow 0^+}z^2\widehat{E}_2|V_d\left(\frac{h}{k}+\frac{iz}{k}\right)=\frac{1}{d^2}\gcd(hd,k)^2.
	\end{equation}
\end{remark}

Using $S$ and $V$ operators on non-holomorphic, modular Eisenstein series $\widehat{E}_2$, we can get holomorphic modular forms. 

\begin{lemma}\label{lem:svlevel}
	For given positive integers $M$ and $d$ and $m\in \Z$, we have, for $\tau \in \H$,
	\begin{enumerate}[label=\textnormal{(\roman*)}]
		\item For $m > 0$,
		\begin{equation}
			\widehat{E}_2|S_{M,m}(\tau)\in \mathcal{M}_2(\Gamma_1(M^2)) \quad\text{and}\quad \widehat{E}_2|S_{M,m}|V_d(\tau) \in \mathcal{M}_2(\Gamma_0(\lcm(4,d))\cap \Gamma_1(M^2)).
		\end{equation}
		
		\item If $m=0$, then $\widehat{E}_2|(dV_d-S_{d,0})(\tau)\in \mathcal{M}_2(\Gamma_0(\lcm(4,d))\cap \Gamma_1(d^2))$. 
	\end{enumerate}
\end{lemma}
This lemma follows by \zcref{prop:finsum} and \cite[Lemma 2.2]{MR4519810}.

Next, we want to investigate the transformation property of 
the operator $S$ applied to $\widehat{E}_2$ for $\Gamma_0(N)$. As most of the congruent theta series we will encounter will be modular on $\Gamma_0(N)$ for some $N$, therefore we look at the  action of $\gamma\in \Gamma_0(N)$ of the modified $\widehat{E}_2|S_{M,m}$ functions. 
\begin{proposition}\label{prop:essasquarem}
	
	Assume $M\in \N$ and $m \in \Z$.
	Then we have $  \widehat{E}_2|S_{M,m}|\gamma (\tau)=\widehat{E}_2|S_{M,a^2m}(\tau)$, for $\gamma\coloneqq\begin{psmallmatrix}
		a&b\\c&d
	\end{psmallmatrix}\in \Gamma_0({M^2})$.
	Moreover,    $\widehat{E}_2|S_{M,m}(\tau)$ is invariant under  $\Gamma_1(M^2)$. 
\end{proposition}
\begin{proof}
	Let $\gamma=\begin{psmallmatrix}
		a&b\\c&d
	\end{psmallmatrix}\in \Gamma_0(M^2)$. Then for $\delta_j=\begin{psmallmatrix}
		1 & \frac{j}{M}\\0&1
	\end{psmallmatrix}$,
	\begin{equation}
		\widehat{E}_2|S_{M,m}| \gamma(\tau) = \frac{1}{M}\sum_{j=0}^{M-1}\widehat{E}_2|\delta_j| \gamma(\tau)\zeta_{M}^{-jm}. 
	\end{equation}
	
	Now take $j'\equiv jd^2\pmod{M^2}$ so that $jd\equiv j'a\pmod{M^2}$. Also, let 
	\begin{align}
		\gamma'\coloneqq \delta_j \gamma \delta_j'^{-1} & = \begin{pmatrix}
			1 & \frac{j}{M}\\0&1
		\end{pmatrix}\begin{pmatrix}
			a&b\\c&d
		\end{pmatrix}\begin{pmatrix}
			1&-\frac{j'}{M}\\0&1
		\end{pmatrix}\\
		& = \begin{pmatrix}
			a+\frac{jc}{M}&b-\frac{jj'c}{M^2}+\frac{jd-j'a}{M}\\ c&d-\frac{j'c}{M}
		\end{pmatrix}. 
	\end{align}
	Note that $M^2 \mid c$ and $a+\frac{jc}{M}\in \Z$ since $\gamma\in \Gamma_0(M^2)$. Moreover, we can have
	\begin{equation}
		d-\frac{j'c}{M}, b-\frac{jj'c}{M^2}+\frac{jd-j'a}{M} \in \mathbb{Z}
	\end{equation}
	for our choice of $j'$. Consequently, 
	\begin{align}
		\widehat{E}_2|S_{M,m}|\gamma (\tau)& = \frac{1}{M}\sum_{j=0}^{M-1}\widehat{E}_2|\delta_j|\gamma (\tau)\zeta_M^{-jm}= \frac{1}{M}\sum_{j'=0}^{M-1}\widehat{E}_2|\gamma' |\delta_{j'}(\tau)\zeta_M^{-a^2j'm}\\
		& =\frac{1}{M}\sum_{j'=0}^{M-1}\widehat{E}_2 |\delta_{j'}(\tau)\zeta_M^{-a^2j'm}= \widehat{E}_2|S_{M,a^2m}(\tau).
	\end{align}
	Observe that, if $\gamma=\begin{psmallmatrix}
		a&b\\c&d
	\end{psmallmatrix}\in \Gamma_1(M^2)$, then $a\equiv 1\pmod{M}$. This completes the proof.
\end{proof}

Using these transformations, we will try to find all the invariant elements of $\mathcal{E}_2(\Gamma_0(N))$ using the $S $ and $V$ operators. First, we recall the structure theorem for the groups $(\mathbb{Z}/2^n\mathbb{Z})^{\times}$ and for odd prime $p$, $(\mathbb{Z}/p^n\mathbb{Z})^{\times}$.

\begin{lemma}\label{lem:z2}  
	For $n\geq 2$, 
	\begin{enumerate}[label=\textnormal{(\roman*)}]
		\item $(\mathbb{Z}/2^n\mathbb{Z})^{\times}\cong \Z_2 \times (\Z/2^{n-2}\Z)$ via the isomorphism $\phi:\Z_2 \times (\Z/2^{n-2}\Z)\rightarrow (\mathbb{Z}/2^n\mathbb{Z})^{\times}$; given by $\phi(i,j)=(-1)^i(5)^j$.
		\item For an odd prime $p$, $(\mathbb{Z}/p^n\mathbb{Z})^{\times}=\langle g \rangle$ or $\langle g+p \rangle$ where $g$ is the generator of the cyclic group $(\mathbb{Z}/p\mathbb{Z})^{\times}.$
	\end{enumerate}
\end{lemma}
	
Next we try finding invariant elements under the action of $\Gamma_0(M^2)$ and get the following. 
\begin{proposition}\label{prop:invcond}
	If $f(\tau)=\sum_{m\in \Z/m\Z} \alpha_m\widehat{E}_2|S_{M,m}(\tau)$ is an invariant element under the action of $\Gamma_0(M^2)$ then 
	\begin{equation}
		\alpha_{a^{2}m}=\alpha_m,
	\end{equation}
	for all $a\in\left(\Z/ M^2\Z\right)^\times$.
\end{proposition}
\begin{proof}
	Let $\gamma=\begin{psmallmatrix}
		a&b\\c&d
	\end{psmallmatrix}\in \Gamma_0(M^2)\subset \SL _2(\Z)$.
	Therefore, $c\equiv0\pmod{M}$ and $ ad\equiv1\pmod{M^2}$ implies $a\in\left(\Z/ M^2\Z\right)^\times$. Let an invariant element  under the action of $\Gamma_0(M^2)$ is given by \begin{equation}
		f(\tau)=\sum_{m\in \Z/m\Z} \alpha_m\widehat{E}_2|S_{M,m}(\tau).
	\end{equation}
	Now, using \zcref{prop:essasquarem} and invariance, if $ f(\tau)=f|\gamma(\tau)$, then,
	\begin{align}
		\sum_{m\in \Z/m\Z} \alpha_m\widehat{E}_2|S_{M,m}(\tau)&= \sum_{m\in \Z/m\Z} \alpha_m\widehat{E}_2|S_{M,m}|\gamma(\tau)\\
		&= \sum_{m\in \Z/m\Z}\alpha_m\widehat{E}_2|S_{M,a^2m}(\tau)\\
		&\stackrel{\mathmakebox[\widthof{=}]{\text{$a^2m\mapsto m$}}}{=} \hspace{10pt} \sum_{m\in \Z/m\Z} \alpha_{a^{-2}m}\widehat{E}_2|S_{M,m}(\tau).
	\end{align}
	Therefore, we get $\alpha_{a^{-2}m}=\alpha_m \text{ which is same as } \alpha_m=\alpha_{a^2m} .$
\end{proof}

\begin{theorem}\label{thm:e2basisthm}
	Let $f=\displaystyle\sum_{m \in\left(\Z/ 2^n\Z\right)^\times}  c_m\widehat{E}_2|_{S_{2^n,m}} \in   M_2(\Gamma_1(2^{2n}))$, for some constants $c_m$. Then $f \in  M_2(\Gamma_0(2^{2n})) $  if and only if  
	\begin{equation}
		f\in \emph{Span}\{\widehat{E}_2|_{S_{2^j,m}}\, :\,  j=1, \ldots  ,\min(3, n),  \, m \in\left(\Z/ 2^j\Z\right)^\times\}.
	\end{equation}
\end{theorem}

\begin{proof}
	From \zcref{prop:invcond}, for a fixed $m\in \left(\Z/2^n\Z\right)^{\times}$,
	\begin{equation}
		c_m=c_{a^2m}, \text{ for all } a\in \left(\Z/2^n\Z\right)^{\times}.
	\end{equation}
	By \zcref{lem:z2}, if $n \geq 3$, the cyclic group $H\coloneqq\{a^2\mid  a \in \left(\Z/2^n\Z\right)^{\times}\}$ is a group of order $2^{n-3}$ in $\left(\Z/2^n\Z\right)^{\times}$ that does not contain $-1$. Hence,
	\begin{equation}
		\left(\Z/2^n\Z\right)^{\times}/ H \cong \left(\Z/8\Z\right)^{\times}=\{1,3,5,7\}.
	\end{equation}
	Hence, if $n \geq 3$, we get
	\begin{equation}
		f\in \mbox{Span}\{\widehat{E}_2|_{S_{8,m}}\ :\ m \in\left(\Z/ 2^j\Z\right)^\times\}.
	\end{equation}

\end{proof}

\begin{remark}\label{rem:oddprimep}
	For the case of odd primes $p$, there are only three independent elements in the group $\mathcal{E}_2(\Gamma_0(p^2))$, namely, $\widehat{E}_2|(S_{p,0}-pV_p), \sum_{\left(\frac{a}{p}\right)=1}\widehat{E}_2|S_{p,a}$ and $\sum_{\left(\frac{b}{p}\right)=-1}\widehat{E}_2|S_{p,b} $. Here, $\left(\frac{x}{p}\right)$ is the Legendre symbol. 
\end{remark}

\begin{example}
	For $n=4$, we check $5^2\equiv 9\pmod{16}, 5^4\equiv 1\pmod{16}$. The second condition makes $c_m=c_{m}$, trivially true. But the first condition makes
	\begin{equation}
		c_1=c_9, \qquad c_3=c_{11}, \qquad c_5=c_{13}, \qquad c_7=c_{15}.
	\end{equation}
	
	Therefore, the invariant elements will be
	\begin{align}
		\widehat{E}_2|S_{16,1}+\widehat{E}_2|S_{16,9}&=\widehat{E}_2|S_{8,1},\\
		\widehat{E}_2|S_{16,3}+\widehat{E}_2|S_{16,11}&=\widehat{E}_2|S_{8,3},\\
		\widehat{E}_2|S_{16,5}+\widehat{E}_2|S_{16,13}&=\widehat{E}_2|S_{8,5},\\
		\widehat{E}_2|S_{16,7}+\widehat{E}_2|S_{16,15}&=\widehat{E}_2|S_{8,7}.
	\end{align}
\end{example}

We finish this section with the following proposition that is useful for later simplifications of Eisenstein elements. 
\begin{lemma}\label{lem:e2v2}
	For $\ell$ odd
	\begin{equation}
		\widehat{E}_2|S_{2^k,\ell}|V_2=\frac{1}{3}\widehat{E}_2|S_{2^{k+1},2\ell}.
	\end{equation}
\end{lemma}
\begin{proof}
	\begin{align}
		\widehat{E}_2|S_{2^k,l}|V_2(\tau)&=\sum_{n\equiv l \pmod {2^k}}\sigma_1(n)q^{2n} \stackrel{\text{m=2n}}{=}\sum_{m\equiv 2l\pmod {2^{k+1}}}\sigma_1\left(\frac{m}{2}\right)q^{m} 
		\\
		&=\sum_{m\equiv 2l\pmod {2^{k+1}}}\left(\frac{\sigma_1(m)}{\sigma_1(2)}\right)q^{m}.      
	\end{align}     
	The last line follows from the fact $\sigma_1$ is multiplicative. As $m$ is odd, we get that
	\begin{equation*}      
		\widehat{E}_2|S_{2^k,l}|V_2(\tau)=\frac{1}{3}\sum_{m\equiv 2l \pmod {2^{k+1}}}\sigma_1(m)q^m=\frac{1}{3}\widehat{E}_2|S_{2^{k+1},2l}(\tau). \qedhere
	\end{equation*}
\end{proof}

\section{Normalized basis of \texorpdfstring{$\mathcal{E}_2(\Gamma_0(p^k))$}{the space of weight 2 Eisenstein series of level p\string^k} using \texorpdfstring{$S$}{S} and \texorpdfstring{$V$}{V} operators}\label{sec:basissection}
In this section, we formulate a method to generate a basis for Eisenstein series space of weight 2 on $\Gamma_0(p^k)$ for primes $p$ up to a certain exponent $k$ by proving \zcref{thm:basisgenerationthm}. The issue for higher exponents is also explained in the process.

Recall that the number of cusps of $\Gamma_0(N)$ and the dimension of the Eisenstein series space on $\G_0(N)$ are closely related. 
Referring back to Proposition 8.5.21 in \cite{book:cohen-stromberg}, we get that 
\begin{equation}\label{eq:dimofEseries}
	\dim (\mathcal{E}_2(\Gamma_0(N)))=e(N)-1,
\end{equation}
where $e(N)$ is the number of non-equivalent cusps of $\Gamma_0(N)$. Furthermore, Proposition 8.5.15 in \cite{book:cohen-stromberg} gives us the formula to calculate $e(N)$. 
\begin{proposition}
	The number of cusps of $\Gamma_0(N)$ is
	\begin{equation}\label{eq:cuspcounting}
		e(N)=\sum_{d\mid N}\phi\left(\gcd\left(d, \frac{N}{d}\right)\right),
	\end{equation}
	where $\phi(n)$ is the \emph{Euler phi function}, counting the natural numbers small than $n$ and co-prime to $n$.  
\end{proposition}

Combining \zcref{eq:dimofEseries,eq:cuspcounting} we get that, for any prime $p$,  
\begin{equation}\label{eq:exactdimension}
	\dim(\mathcal{E}_2(\G_0(p^k)))=\begin{cases}
		(p+1)p^{\frac{k}{2}-1}-1, & \text{if } k \text{ is even,} \\
		2p^{\frac{k-1}{2}}-1      & \text{if } k \text{ is odd.}
	\end{cases}
\end{equation}

\begin{proof}[Proof of \zcref{thm:basisgenerationthm}]
	Observe that, for any $k\in \N\cup \{0\}$, the newforms of $\mathcal{E}_2(\G_0(p^{2k+1}))$ will be of the form $\widehat{E}_2|S_{p^{k+1},m}$ for $m \in (\Z/p^{k+1}\Z)^{\times}$, as $\widehat{E}_2|S_{p^{k},m}\in \mathcal{E}_2(\G_0(p^{2k})) $ for all $m \in (\Z/p^{k}\Z)^{\times}  $ by \zcref{lem:svlevel}. Moreover,  \zcref{lem:svlevel} gives that a newform  of the type $\widehat{E}_2|S_{p^{k+1},m} \in \mathcal{E}_2(\G_0(p^{2k+2} ))$. Therefore, 
	$\mathcal{M}_2(\G_0(p^{2k+1}))$ will not have any newform. Furthermore, from \zcref{eq:exactdimension}, the size of the set  $\{f,f|V_p: f \text{ is a basis element of } \mathcal{E}_2(\G_0(p^{2k})) \} $, is larger than $\dim(\mathcal{E}_2(\G_0(p^{2k+1})))$. Therefore, we get a basis for $\dim(\mathcal{E}_2(\G_0(p^{2k+1})))$ from the spanning set $\{f,f|V_p: f \text{ is a basis element of } \mathcal{E}_2(\G_0(p^{2k})) \} $. For this reason, we focus only on finding a basis for $\mathcal{E}_2(\G_0(p^{2k}))$.
	
	\paragraph{Case 1: $p=2$}
	Following \zcref{thm:e2basisthm}, we can only make a complete basis  for $\mathcal{E}_2(\Gamma_0(2^{2k}))$ with the help of the operators $S$ and $V$ when $k\leq 3$. We start with the newforms of the type $\widehat{E}_2|S_{2^k,m}(\tau)$ for $ m \in (\Z/2^k\Z)^{\times}$, there are $2^{k-1}$ such elements. Next, we include the oldform $\widehat{E}_2|(2V_2-S_{2,0})|V_{2^{2k-2}}(\tau)$, which is the only element with a nonzero constant term in its Fourier expansion. Finally, we take oldforms from the set $\{f, f|V_2 : f \text{ is a basis element of } \mathcal{E}_2(\G_0(p^{2k-1}))\}$. Thus, we get $2(2\cdot 2^{k-1}-1)+1+2^{k-1}$ elements in total. Since $\dim(\mathcal{E}_2(\Gamma_0(2^{2k})))=3\cdot2^{k-1}-1$, we can get a basis for $\mathcal{E}_2(\Gamma_0(2^{2k}))$ from the spanning set mentioned above.
	
	\paragraph{Case 2: $p$ odd}
	Following \zcref{rem:oddprimep}, there are only three independent elements in $\mathcal{E}_2(\Gamma_0(p^2))$ when $p$ is an odd prime. However, by \zcref{eq:exactdimension}, $\dim (\mathcal{E}_2(\Gamma_0(p^2)))=p$. Therefore, we cannot generate the entire basis set for $\mathcal{E}_2(\Gamma_0(p^2))$ using $S,V$ operators for $p>3$. For $p=3$, a basis for $\mathcal{E}_2(\Gamma_0(3^2))$ is given by the elements $\widehat{E}_2|S_{3,m}(\tau)$ for $ m\in (\Z/3\Z)^{\times}$ and the element $\widehat{E}_2|(3V_3-S_{3,0})(\tau)$.
\end{proof}

\begin{remark}
	A list of indexed basis elements is given in \ref{app:evenbasis} and \ref{app:oddbasis} at the end. We are going to use these basis elements to extract the Eisenstein series parts of our congruent theta series in the next four sections. 
	
\end{remark}

\section{The theta series \texorpdfstring{$\theta_{k,2}$}{theta\_k,2}}\label{sec:theta2}
In this section, we work with the congruent theta series $\theta_{k,2}$.  
We will use the basis in \ref{app:evenbasis} to obtain the Eisenstein series part of $\theta_{k,2}$ using the values at the cusps. 

\subsection{Analysis of \texorpdfstring{$\theta_{0,2}$}{theta\_0,2}}
Notice that from \zcref{eq:jacobitheta},
\begin{equation}
	\theta_3(8\tau)^4=\left(\sum_{n \in \Z}q^{8\frac{n^2}{2}}\right)^4=\left(\sum_{n \in \Z}q^{(2n)^2}\right)^4=\theta_{0,2}(\tau).
\end{equation}
Recall from \cite{MR3848417}, $\theta_3(2\tau)^4= \EE 4 1(\tau)+8\EE 4 2 (\tau)$, and hence 

\begin{equation}\label{eq:theta02}
	\theta_{0,2}(\tau)=( \EE 4 1+8\EE 4 2)|V_4(\tau)=\EE {16} 1 (\tau)+8\EE {16} 5(\tau).
\end{equation}
Notice that the cusp form part in   $\theta_{0,2}$ is 0. Therefore, looking at the Eisenstein elements and noticing the operators applied, we get the following result.

\begin{proposition}
	For $n \in \N \cup \{0\}$, 
	
	\begin{equation}
		r_{0,2}(n)=
		\begin{cases}
			1,                                                                                 & \textnormal{if } n = 1,                      \\
			24\left[\sigma_1\left(\frac{n}{4}\right)-2\sigma_1\left(\frac{n}{8}\right)\right], & \textnormal{if } n \equiv 0 \pmod{8}, n > 1, \\
			8\sigma_1\left(\frac{n}{4}\right),                                                 & \textnormal{if } n \equiv 4 \pmod{8},        \\
			0,                                                                                 & \textnormal{otherwise.}
		\end{cases}
	\end{equation}  
\end{proposition}

\subsection{Analysis of \texorpdfstring{$\theta_{1,2}$}{theta\_1,2}}
Notice that
\begin{equation}
	\theta_{1,2}(\tau)=16q^4+64q^{12}+96q^{20}+128q^{28}+O(q^{36}).
\end{equation}
Again, using \zcref{thm:thetamod}, we get that,
\begin{equation}
	\theta_{1,2}(\tau)\in \mathcal{M}_2(\Gamma_0(16)\cap \Gamma_1(2)).
\end{equation}
We get the following values at the cusp for $\theta_{1,2}(\tau).$
\begin{center}
	\begin{tabular}{ |c|c|c|c|c|c|c| } 
		\hline

		cusps of $\G_0(16)$ & $\frac{1}{8}$& $\frac{1}{4}$  & $\frac{3}{4}$  & $\frac{1}{2}$ &  0 &$\infty$ \\
		\hline\hline
		
		values of $\theta_{1,2}$ & 1&$-\frac{1}{4}$&$-\frac{1}{4}$&$-\frac{1}{16}$&$-\frac{1}{64}$&0\\ \hline
	\end{tabular}
\end{center}

Therefore, comparing the cusp values from the table in \ref{app:evenbasis}, we get
\begin{equation}\label{eq:theta12}
	\theta_{1,2}(\tau)=16\EE {16} 5(\tau).
\end{equation}
Again, there is no cusp form part in  $\theta_{1,2}$. Hence, looking at the operators applied in the Eisenstein series part, we get the following result.
\begin{proposition}\label{prop:r12} 
	For any $n \in \N\cup \{0\}$,
	\begin{equation}
		r_{1,2}(n)=\begin{cases}
			16\sigma_1\left(\frac{n}{4}\right), & \textnormal{if }n \equiv4\pmod{8}, \\
			0, \hspace{15pt}                    & \textnormal{otherwise.}
		\end{cases}
	\end{equation}
\end{proposition}

\section{Theta series \texorpdfstring{$\theta_{k,4}(\tau)$}{theta\_k,4 (tau)}}\label{sec:theta4} 
In this section, we show a detailed analysis of the congruent theta series $\theta_{k,4}(\tau)$. We decompose $\theta_{k,4}$ into Eisenstein series and the cusp form parts. Thereafter, the Fourier expansions of these parts help us to precisely formulate the Fourier coefficients of $\theta_{k,4}(\tau)$ which is exactly the number of representations of natural numbers by the congruent quadratic form $Q_{k,4}$. Also, all the Fourier expansions in these sections can be obtained by \cite{maple} or any other computer algebra system.  We start this section with $k=1$ which is the same for $k=3$.

\subsection{Analysis of \texorpdfstring{$\theta_{1,4}(\tau)=\theta_{3,4}(\tau)$}{theta\_1,4 = theta\_3,4}} 
Notice that
\begin{align}
	\theta_{1,4}(\tau)&=\sum_{\substack{x_i\in \mathbb{Z} \\ x_i\equiv1(4)}} q^{x_1^2+x_2^2+x_3^2+x_4^2}=  \sum_{\substack{x_i\in \mathbb{Z}\\x_i\equiv3(4)}}q^{(-x_1)^2+(-x_2)^2+(-x_3)^2+(-x_4)^2} =  \theta_{3,4}(\tau),\\
	&=q^4+4q^{12}+6q^{20}+8q^{28}+13q^{36}+O(q^{44}).
\end{align}
Moreover, 
\begin{equation}
	\theta_{1,2}(\tau)= \left(\sum_{m\in \Z}q^{(4m+1)^2}+q^{(4m+3)^2}\right)^4=\left(2\sum_{m\in \Z}q^{(4m+1)^2}\right)^4=16\theta_{1,4}(\tau).
\end{equation}	
We have an immediate result for $r_{1,4}(n)=r_{3,4}(n)$ as follows, by \zcref{prop:r12}. 
\begin{proposition}
	For any $n \in \N\cup \{0\}$, 
	\begin{equation}
		r_{1,4}(n)=r_{3,4}(n)=\begin{cases}
			\sigma_1\left(\frac{n}{4}\right), & \text{if }n\equiv 4\pmod 8, \\
			0,                                & \text{otherwise.}
		\end{cases}
	\end{equation}

\end{proposition} 
\subsection{Analysis of \texorpdfstring{$\theta_{0,4}$}{theta\_0,4} and \texorpdfstring{$\theta_{2,4}$}{theta\_2,4}}

\begin{lemma} For the congruent theta series $\theta_{k,4}(\tau)$ for $k=0,2$, we have the following decomposition. 
	\begin{align}
		\theta_{2,4}(\tau)&=16\EE {64}{11}(\tau)=-\frac{2}{3}\widehat{E}_2|S_{2,1}|V_{16}(\tau).\\
		\theta_{0,4}(\tau)&=\EE{64}{1}(\tau)+8\EE{64}{11}(\tau)= \widehat{E}_2|(2V_2-S_{2,0})|V_{16}(\tau)-\frac{1}{3}\widehat{E}_2|S_{2,1}|V_{16}(\tau).
	\end{align}

\end{lemma}

\begin{proof}
	Observe that
	\begin{equation}
		\theta_{1,2}(4\tau)=\left(\sum_{n \in \Z}q^{4\cdot (2n+1)^2}\right)^4=\left(\sum_{n \in \Z}q^{ (4n+2)^2}\right)^4=\theta_{2,4}(\tau).
	\end{equation}
	From \zcref{eq:theta12}, we can get
	\begin{equation}
		\theta_{2,4}(\tau)=16 \EE{64}{11}=-\frac{2}{3}\widehat{E}_2|S_{2,1}|V_{16}(\tau). 
	\end{equation}
	Further notice that
	\begin{equation}
		\theta_{0,2}(4\tau)=\left(\sum_{n \in \Z}q^{4\cdot (2n)^2}\right)^4=\left(\sum_{n \in \Z}q^{ (4n)^2}\right)^4=\theta_{0,4}(\tau).
	\end{equation}
	Therefore, from \zcref{eq:theta02}, we get that
	\begin{equation}
		\theta_{0,4}(\tau)=\EE{64}{1}(\tau)+8\EE{64}{11}(\tau).
	\end{equation}


	
	
	
\end{proof}
Looking at the decomposition of the congruent theta series $\theta_{2,4}(\tau)$ and $\theta_{0,4}(\tau)$, we get the following proposition. 

\begin{proposition} For any $n \in \N\cup \{0\}$,
	\begin{equation}
		r_{2,4}(n)=\begin{cases}
			16 \sigma_1\left(\frac{n}{16}\right), & \text{if } n\equiv 16\pmod{32}, \\
			0,                                    & \text{otherwise.}
		\end{cases}    
	\end{equation}
	
	\begin{equation}
		r_{0,4}(n)=\begin{cases}
			1,                                                                                    & \textnormal{if } n=1,                  \\
			24\left[\sigma_1\left(\frac{n}{16}\right)-2\sigma_1\left(\frac{n}{32}\right) \right], & \textnormal{if } n\equiv0\pmod{32},    \\
			8\sigma_1\left(\frac{n}{16}\right),                                                   & \textnormal{if } n \equiv 16\pmod{32}, \\
			0,                                                                                    & \textnormal{otherwise.}
		\end{cases}
	\end{equation}

\end{proposition}

\section{Theta series \texorpdfstring{$\theta_{k,3}(\tau)$}{theta\_k,3 (tau)}}\label{sec:theta3}
In this section, we work the congruent quadratic form modulo 3. We want to investigate about the integers represented by $Q_{k,3}$. From \zcref{thm:thetamod}, we get
\begin{equation}
	\theta_{k,3}\in \mathcal{M}_2(\Gamma_0(36)\cap \Gamma_1(3))
	= \bigoplus_{\chi(-1) = 1 }\mathcal{M}_2(\Gamma_0(36),\chi)
	=\mathcal{M}_2(\Gamma_0(36)). 
\end{equation}
Now, $\dim(\mathcal{M}_2(\Gamma_0(36)))=12 $, $ \dim({\mathcal{S}_2(\Gamma_0(36))})=1$. Therefore, $ \dim(\mathcal{E}_2(\Gamma_0(36)))=11$. In order to a generate basis for the Eisenstein series space, we list all the elements of  $\mathcal{E}_2(\Gamma_0(36))$ with the help of $S$ and $V$ operators, starting with the element with element with constant term (here the term $c=-\frac{1}{24}$ is for normalization).	
\begin{gather}
	B_1\coloneqq \widehat{E}_2|(2V_2-S_{2,0})|V_9(\tau)=2\widehat{E}_2|V_{18}-\widehat{E}_2|S_{2,0}|V_9(\tau),\\
	\begin{alignedat}{3}
		B_2 &\coloneqq c\widehat{E}_2|S_{6,1}(\tau),\hspace{30pt} &
		B_3 &\coloneqq c\widehat{E}_2|S_{6,5}(\tau), & & \\
		B_4 &\coloneqq c\widehat{E}_2|S_{3,1} (\tau),\hspace{30pt} &
		B_5 &\coloneqq c\widehat{E}_2|S_{3,1}|V_2 (\tau),\hspace{30pt} &
		B_6 &\coloneqq c\widehat{E}_2|S_{3,1}|V_4(\tau),\\
		B_7 &\coloneqq c\widehat{E}_2|S_{3,2} (\tau),\hspace{30pt} &
		B_8 &\coloneqq c\widehat{E}_2|S_{3,2}|V_2(\tau), \hspace{30pt} &
		B_9 &\coloneqq c\widehat{E}_2|S_{3,2}|V_4(\tau),\\
		B_{10} &\coloneqq c\widehat{E}_2|S_{2,1} (\tau),\hspace{30pt} &
		B_{11} &\coloneqq c\widehat{E}_2|S_{2,1}|V_3(\tau), \hspace{30pt} &
		B_{12} &\coloneqq c\widehat{E}_2|S_{2,1}|V_9(\tau).
	\end{alignedat}
\end{gather}

As $\dim({\mathcal{E}_2(\Gamma_0(36))})=11$, some elements of the above list must be linearly dependent. And we can easily check that,
\begin{equation}\label{eq:linrel}
	B_2+B_3+4B_{11}-3B_{12}-B_{9}=0.
\end{equation}
So, we will take $\mathcal{E}_2(\Gamma_0(36))=\C[ B_1,\dots,B_{11}]$. The (non-equivalent) cusps of $\Gamma_0(36)$ are as follows: $0$, $\frac{1}{18}$, $\frac{1}{12}$, $\frac{1}{9}$, $\frac{1}{6}$, $\frac{1}{4}$, $\frac{1}{3}$, $\frac{5}{12}$, $\frac{1}{2}$, $\frac{2}{3}$, $\frac{5}{6}$, $\infty$. Using \zcref{eq:ES,eq:ESV,eq:EV}, we can get the values at cusps for all the basis elements of $\mathcal{E}_2(\Gamma_0(36))$. A table of the cusp values can be found in \ref{app:gamma036table}. Having set the background, we will now deal with $\theta_{k,3}(\tau)$ for different values of $k$.
\subsection{Analysis of \texorpdfstring{$\theta_{0,3}$}{theta\_0,3}} 
Notice that, 
\begin{equation}
	\theta_3(18\tau)^4=\left(\sum_{n \in \Z}q^{18\frac{n^2}{2}}\right)^4=\left(\sum_{n \in \Z}q^{(3n)^2}\right)^4=\theta_{0,3}(\tau).
\end{equation}
Recall, $\theta_3(2\tau)^4= \EE{4}{1}(\tau)+8\EE{4}{2}(\tau)$, and hence 
\begin{align}\label{eq:theta03}
	\theta_{0,3}(\tau)
	&=(\EE{4}{1}+8\EE{4}{2})|V_9(\tau)
	=2\widehat{E}_2|V_{18}(\tau)-\widehat{E}_2|S_{2,0}|V_9(\tau)-\frac{1}{3}\widehat{E}_2|S_{2,1}|V_9(\tau).
\end{align}
As a result, we get the following proposition.
\begin{proposition}\label{prop:r03}
	For any $n \in \N\cup \{0\}$,  
	\begin{equation}
		r_{0,3}(n)=\begin{cases}
			1,                                                                                   & \textnormal{if } n=1,                 \\
			24\left[\sigma_1\left(\frac{n}{9}\right)-2\sigma_1\left(\frac{n}{18}\right) \right], & \textnormal{if } n\equiv0\pmod{18},   \\
			8\sigma_1\left(\frac{n}{9}\right),                                                   & \textnormal{if } n \equiv 9\pmod{18}, \\
			0,                                                                                   & \textnormal{otherwise.}
		\end{cases}
	\end{equation}
	
\end{proposition}

\subsection{Analysis of \texorpdfstring{$\theta_{1,3}$}{theta\_1,3}}

This is quite a striking case. In this case, we will encounter that $\theta_{1,3}(\tau)$ has a cusp form part. We first take a look at the Fourier series of $\theta_{1,3}(\tau)$ itself.
\begin{equation}
	\theta_{1,3}(\tau)=q^4+4q^7+6q^{10}+4q^{13}+q^{16}+4q^{19}+12q^{22}+12q^{25}+\cdots 
\end{equation}
\begin{remark}
	Clearly, $r_{1,3}(n)=0$ unless $n\equiv 1,4\pmod{6}$. 
\end{remark}

We have the values of the basis elements of $\mathcal{E}_2(\Gamma_0(36))$ at the cusps of $\Gamma_0(36)$ in \ref{app:gamma036table}. We take the following linear combination at a cusp $c$. 
\begin{equation}
	\theta_{1,3}(\tau)=\sum_{i=1}^{11}\ell_iB_i.
\end{equation}
Solving the system of linear equations, we get the following solution set:
\begin{equation}
	\ell_2=\frac{1}{3}, \ell_6=-2, \ell_8=1.
\end{equation}
Assume we have the following Eisenstein series and cusp form decomposition,
\begin{equation}
	\theta_{1,3}(\tau)=E_{1,3}(\tau)+f_{1,3}(\tau).
\end{equation}
Then we will get the Eisenstein series component can be expressed as follows. 
\begin{equation}
	E_{1,3}=-\frac{1}{24}\left(\frac{1}{3}\widehat{E}_2|S_{6,1}-2\widehat{E}_2|S_{3,1}|V_4+\widehat{E}_2|S_{3,2}|V_2\right)    .
\end{equation}
Thereafter, we extract the cusp form part, which is
\begin{align}
	f_{1,3}(\tau)&=\theta_{1,3}(\tau)-E_{1,3}(\tau)\\
	\label{eq:cusp}&=-\frac{1}{3}(q-4q^7+2q^{13}+8q^{19}-5q^{25}-4q^{31}-10q^{37}+8q^{43}+9q^{49}+\cdots)   
\end{align} 
Moreover, 
\begin{equation}
	f_{1,3}(\tau) \in \mathcal{S}_2(\Gamma_0(36)).
\end{equation}
It is known that $\dim(\mathcal{S}_2(\Gamma_0(36)))=1$ and is generated by the normalized Hecke eigenform $\eta(6\tau)^4$, where $\eta(\tau)$ is the \emph{Dedekind eta function}. Finally, comparing the Fourier series, we get that,

\begin{equation}
	-3f_{1,3}(\tau)=\eta(6\tau)^4.
\end{equation}

\par
Summing up all the discussions above, we get the following proposition. 
\begin{proposition}\label{prop:propforECPC}
	For $n \in \N\cup \{0\}$,       
	\begin{equation}
		r_{1,3}(n)=
		\begin{cases}
			\sigma_1\left(\frac{n}{2}\right)-2\sigma_1\left(\frac{n}{4}\right), & \textnormal{if } n\equiv4\pmod{12},  \\
			\sigma_1\left(\frac{n}{2}\right)                                    & \textnormal{if } n\equiv10\pmod{12}, \\
			\frac{1}{3}\sigma_1(n)-\frac{1}{3}a(n),                             & \textnormal{if } n\equiv1\pmod{6},   \\
			0,                                                                  & \textnormal{otherwise,}
		\end{cases}
	\end{equation}
	where $\eta(6\tau)^4=\sum_{n\geq 1}a(n)q^n$.
\end{proposition}

\begin{remark}
	Notice that, $\theta_{1,3}(\tau)=\theta_{2,3}(\tau)$. Therefore, for any $n \in \N$, $r_{1,3}(n)=r_{2,3}(n)$. 
\end{remark}

\section{Theta series \texorpdfstring{$\theta_{k,6}(\tau)$}{theta\_k,6 (tau)}}\label{sec:theta6}
In this section, we will deal with the congruent quadratic form modulo 6. We want to investigate the integers represented by $Q_{k,6}$. 
From \zcref{thm:thetamod},
\begin{equation}
	\theta_{k,6}\in \mathcal{M}_2(\Gamma_0(144)\cap \Gamma_1(6))= \bigoplus_{\chi(-1)=1 }\mathcal{M}_2(\Gamma_0(144),\chi) =\mathcal{M}_2(\Gamma_0(144)). 
\end{equation}

By using \zcref{eq:dimofEseries,eq:cuspcounting}, we get that $\dim(\mathcal{E}_2(\Gamma_0(144)))=23$. We list all the possible elements that span $\mathcal{E}_2(\Gamma_0(144))$ as follows.
\begin{gather}\allowdisplaybreaks[0]
	A_1\coloneqq{}\widehat{E}_2|(2V_2-S_{2,0})|V_{36}(\tau)=2\widehat{E}_2|V_{72}-\widehat{E}_2|S_{2,0}|V_{36}(\tau),\\[10pt]
	A_2\coloneqq{}c\widehat{E}_2|S_{12,1}(\tau),\hspace{15pt} A_3\coloneqq{}c\widehat{E}_2|S_{12,5}(\tau),\hspace{15pt} A_4\coloneqq{}c\widehat{E}_2|S_{12,7}(\tau),\hspace{15pt} A_5\coloneqq{}c\widehat{E}_2|S_{12,11}(\tau),\\[10pt]
	A_6\coloneqq{}c\widehat{E}_2|S_{6,1}(\tau),\hspace{15pt} A_7\coloneqq{}c\widehat{E}_2|S_{6,1}|V_2(\tau),\hspace{15pt} A_8\coloneqq{}c\widehat{E}_2|S_{6,1}|V_4(\tau),\\[10pt]
	A_9\coloneqq{}c\widehat{E}_2|S_{6,5}(\tau),\hspace{15pt} A_{10}\coloneqq{}c\widehat{E}_2|S_{6,5}|V_2(\tau),\hspace{15pt} A_{11}\coloneqq{}c\widehat{E}_2|S_{6,5}|V_4(\tau),\\[10pt]
	A_{12}\coloneqq{}c\widehat{E}_2|S_{3,1}(\tau),\hspace{15pt} A_{13}\coloneqq{}c\widehat{E}_2|S_{3,1}|V_2(\tau),\hspace{15pt} A_{14}\coloneqq{}c\widehat{E}_2|S_{3,1}|V_4(\tau),\\
	A_{15}\coloneqq{}c\widehat{E}_2|S_{3,1}|V_8(\tau),\hspace{15pt} A_{16}\coloneqq{}c\widehat{E}_2|S_{3,1}|V_{16}(\tau),\\[10pt]
	A_{17}\coloneqq{}c\widehat{E}_2|S_{3,2}(\tau),\hspace{15pt} A_{18}\coloneqq{}c\widehat{E}_2|S_{3,2}|V_2(\tau),\hspace{15pt} A_{19}\coloneqq{}c\widehat{E}_2|S_{3,2}|V_4(\tau),\\
	A_{20}\coloneqq{}c\widehat{E}_2|S_{3,2}|V_8(\tau),\hspace{15pt} A_{21}\coloneqq{}c\widehat{E}_2|S_{3,2}|V_{16}(\tau),\\[10pt]
	A_{22}\coloneqq{}c\widehat{E}_2|S_{2,1}(\tau),\hspace{15pt} A_{23}\coloneqq{}c\widehat{E}_2|S_{2,1}|V_2(\tau),\hspace{15pt} A_{24}\coloneqq{}c\widehat{E}_2|S_{2,1}|V_3(\tau),\\
	A_{25}\coloneqq{}c\widehat{E}_2|S_{2,1}|V_4(\tau),\hspace{15pt} A_{26}\coloneqq{}c\widehat{E}_2|S_{2,1}|V_{6}(\tau),\\[10pt]
	A_{27}\coloneqq{}c\widehat{E}_2|S_{2,1}|V_{9}(\tau),\hspace{15pt}
	A_{28}\coloneqq{}c\widehat{E}_2|S_{2,1}|V_{12}(\tau),\hspace{15pt}
	A_{29}\coloneqq{}c\widehat{E}_2|S_{2,1}|V_{18}(\tau), \\
	A_{30}\coloneqq{}c\widehat{E}_2|S_{2,1}|V_{36}(\tau).
\end{gather}
Using the \zcref{eq:ES,eq:ESV,eq:EV}, we can get the values at cusps (non-equivalent) for all the basis elements of $\mathcal{E}_2(\Gamma_0(144))$.
Now we start working on different congruent theta series $\theta_{k,6}$ for different values of $k$.

We first consider the theta series $\theta_{0,6}$, $\theta_{2,6}$, and $\theta_{3,6}$.  
\begin{proposition} 
	For $n \in \N\cup\{0\}$,  
	\begin{equation}
		r_{0,6}(n)=
		\begin{cases}
			1                                                                                    & \textnormal{if } n=1,                    \\
			24\left[\sigma_1\left(\frac{n}{36}\right)-2\sigma_1\left(\frac{n}{72}\right)\right], & \textnormal{if } n\equiv \textnormal{0\pmod{72}},\ n>1, \\
			8\sigma_1\left(\frac{n}{36}\right),                                                  & \textnormal{if } n\equiv \textnormal{36\pmod{72}},    \\
			0, \hspace{15pt}                                                                     & \textnormal{otherwise. }
		\end{cases}
	\end{equation}
	\begin{equation}
		r_{1,6}(n)=r_{5,6}(n)=
		\begin{cases}
			\frac{2}{3}\sigma_1\left(\frac{n}{4}\right)+\frac{1}{3}a_{\eta(6\tau)^4}\left(\frac{n}{4}\right), & \textnormal{if } n\equiv 4\pmod{24}, \\
			0,                                                                                                & \textnormal{otherwise. }
		\end{cases}
	\end{equation}
	\begin{equation}
		r_{2,6}(n)=r_{4,6}(n)=
		\begin{cases}
			\frac{1}{3}\sigma_1\left(\frac{n}{4}\right)-\frac{1}{3}a_{\eta(6\tau)^4}\left(\frac{n}{4}\right), & \textnormal{if } n\equiv 4\textnormal{ or 28\pmod{48}}, \\
			-\frac{4}{3}\sigma_1\left(\frac{n}{16}\right)+\frac{1}{3}\sigma_1\left(\frac{n}{4}\right),        & \textnormal{if }n\equiv \textnormal{16\pmod{48}},             \\
			\frac{1}{3}\sigma_1\left(\frac{n}{4}\right),                                                      & \textnormal{if }n\equiv \textnormal{40\pmod{48}},             \\
			0,                                                                                                & \textnormal{otherwise. }
		\end{cases}
	\end{equation}
	\begin{equation}
		r_{3,6}(n)=
		\begin{cases}
			16\sigma_1\left(\frac{n}{36}\right), & \textnormal{if } n \equiv \textnormal{36\pmod{72}}, \\
			0,                                   & \textnormal{otherwise.}
		\end{cases}
	\end{equation}
	where $a_{\eta(6\tau)^4}(k)$ is the $k^{th}$ coefficient of the $q$-expansion of the cusp form $\eta(6\tau)^4$. 
\end{proposition}
\begin{proof}
	Notice that
	\begin{align}
		\theta_{0,6}(\tau)&=\sum_{\substack{\Vec{x}\in \Z^4 \\ x_i\equiv 0\pmod{6}}} q^{x_1^2+x_2^2+x_3^2+x_4^2}
		=\sum_{\Vec{y}\in \Z^4}q^{36y_1^2+36y_2^2+36y_3^2+36y_4^2}\\
		&=\sum_{\Vec{y}\in \Z^4}q^{4(9y_1^2+9y_2^2+9y_3^2+9y_4^2)}=\sum_{\substack{\Vec{t}\in \Z^4 \\ t_i\equiv 0\pmod{3}}} q^{4(t_1^2+t_2^2+t_3^2+t_4^2)}\\
		&=\theta_{0,3}(4\tau).
	\end{align}
	Similarly, we can show that 
	\begin{equation}
		\theta_{2,6}(\tau)=\theta_{1,3}(4\tau), \quad   \theta_{3,6}(\tau)=\theta_{1,2}(9\tau).
	\end{equation}
	The results follow from \zcref{prop:r03,prop:propforECPC,prop:r12} immediately. 
	
	For $\theta_{1,6}(\tau)$, notice the Fourier expansion is, 
	\begin{equation}
		\theta_{1,6}(\tau)=q^4+4q^{28}+10q^{52}+16q^{76}+19q^{100}+20q^{124}+22q^{148}+\cdots
	\end{equation}
	Clearly from the definition, $r_{1,6}(n)=0$ unless $n\equiv 4 \pmod{24}$. 
	Moreover, comparing the values at the cusps of $\Gamma_0(144)$, we get that,
	\begin{equation}
		\theta_{1,6}(\tau)=\frac{2}{3}A_8+\frac{1}{3}f|V_4(\tau),
	\end{equation}
	where $f(\tau)=\eta(6\tau)^4$. 
	Hence, we get the result. 
	
\end{proof}

\section{Relation between point counting on Elliptic curve and theta series}\label{sec:sectionforECPC}

In this section, we will discuss deeply about the results of $\theta_{1,3}(\tau)$ of \zcref{sec:theta3}. Recall, $\theta_{1,3}(\tau)$ has cusp form part $-\frac{1}{3}\eta(6\tau)^4$. Therefore, first we will need to analyze the cusp form $\eta(6\tau)^4$. Later in this section, we will show some more implications coming from our main results mentioned in the introduction section.

\subsection{The Elliptic Curve \texorpdfstring{$y^2=x^3+1$}{y\string^2 = x\string^3 + 1}.}

We begin by examining the geometric curve of genus 1 associated with our weight 2 cusp form $\eta(6\tau)^4$. The corresponding elliptic curve of $\eta(6\tau)^4$ has isogeny class of label 36.a in \cite{lmfdb}. Notably, the modular curve $X_0(36)$ has genus 1 and has a canonical model defined over $\Q$, which is $\Q$-isogeny to the elliptic curve $E:y^2=x^3+1$.   As $E$ is defined over $\Q$, we can take reductions over various primes $p$ and prime powers $p^k$ enabling point counting over finite fields $\F_p$ and $\F_{p^k}$. We will use the following properties of the elliptic curve $E$.   

\begin{proposition}{\cite{book:silverman}}\label{prop:ellipticcurveprop}
	For the elliptic curve $E:y^2=x^3+1$, we have the following,
	\begin{enumerate}[label=\textnormal{(\arabic*)}]
		\item $E(\Q)_{\text{tor}}\cong \Z/6\Z$,
		\item  $E$ has CM by $-3$,
		\item $L(E,s)=L(\eta(6\tau)^4,s)$.
	\end{enumerate}
\end{proposition}

For being associated with the CM elliptic curve $E$, $\eta(6\tau)^4$ is also a \emph{CM cusp form} and as a consequence we get that, the $a_p$ coefficients of L-function of $\eta(6\tau)^4$ will be 0 when any prime $p$ is inert in $\Q(\sqrt{-3})$, that is, when $p\equiv 5\pmod{6}$.

\subsection{Proof of \textnormal{\zcref{thm:r13ecpc}}}
We will need the following lemmas and propositions for our proof. First we recall the following theorem to get a formula for point counting on elliptic curve $E$ over the finite field $\F_{p^k}$, denoted as $N_p^E(k)$.    
\begin{theorem}{(Weil)}\label{thm:weil}
	
	Let $a_p$ be the $p^{th}$ coefficient of the $L$-function of the elliptic curve $E$. 
	Let $\alpha$ and $\beta$ be the roots of the quadratic equation $x^2-a_px+p=0$. Then for any $ k \in \N$,we 
	have
	\begin{equation}
		N_p^E(k)=p^k+1-\alpha^k-\beta^k. 
	\end{equation}
	
\end{theorem}

\begin{theorem}\label{thm:recursionap} 
	Let $a_n$ be the $n$-th Fourier coefficient of the cusp form $\eta(6\tau)^4$. Then we have the following three-term recursive relation for $a_{p^k}$ for primes $p$ and $k \in \N$:
	\begin{equation}
		a_{p^k}=a_pa_{p^{k-1}}-p \cdot a_{p^{k-2}}. 
	\end{equation}  
\end{theorem}
We prove the following expression of $a_{p^k}$ in terms of powers of $\a$ and $\b$.

\begin{lemma}\label{lem:recursionhelp1}
	We have the following sum formula for $a_{p^k}$,
	\begin{equation}
		a_{p^k}=\sum_{i=0}^k
		\a^i\b^{k-i}.    \end{equation}
\end{lemma}

\begin{proof}
	We will use induction on $k$. For $k=1$, \zcref{thm:weil} gives us $a_p=\a+\b$.
	Next, for $k=2$, using the recursion relation in \zcref{thm:recursionap}, we get that
	\begin{equation}
		a_{p^2}=a_p^2-pa_1=(\a+\b)^2-\a\b=\a^2+\a\b+\b^2.
	\end{equation}

	Now, assume the formula holds for all $k$ up to $k=m$. Then, for $k=m$ and $k=m-1$, we will have,
	\begin{equation}\label{eq:inductionapm1}
		a_{p^{m-1}}=\sum_{i=0}^{m-1} \a^i\b^{m-i-1}, 
	\end{equation}	
	\begin{equation}\label{eq:inductionapm}
		a_{p^m}=\sum_{i=0}^m \a^i\b^{m-i}.
	\end{equation}
	Finally, for $k=m+1$, using \zcref{thm:recursionap} and \zcref{eq:inductionapm,eq:inductionapm1},  and we will have
	\begin{align}
		a_{p^{m+1}}&=a_pa_{p^m}-p\cdot a_{p^{m-1}}\\
		&=a_p\sum_{i=0}^m \a^i\b^{m-i}-p\sum_{i=0}^{m-1} \a^i\b^{m-i-1}\\
		&=(\a+\b)\cdot \sum_{i=0}^m \a^i\b^{m-i}-\a\b\cdot\sum_{i=0}^{m-1} \a^i\b^{m-i-1}\\
		&=\sum_{i=0}^{m+1} \a^i\b^{m-i+1}.
	\end{align}
\end{proof}

As an immediate corollary and the fact 
\begin{equation}
	\sigma_1(p^k)-p\sigma_1(p^{k-2})=p^k+1,
\end{equation}
we have the following relation which leads to a three-term relation for $r_{1,3}(p^k)$. 
\begin{corollary}\label{cor:recursionhelp2}
	\begin{equation}
		N_p(k)=\sigma_1(p^k)-p\sigma_1(p^{k-2})-(\a^k+\b^k).
	\end{equation}
\end{corollary}

%
%

\begin{proof}[Proof of \zcref{thm:r13ecpc}]
	Using parts 1 and 3 of \zcref{prop:ellipticcurveprop} and \zcref{prop:propforECPC}, for $p\equiv 1\pmod{6}$, we get,
	\begin{equation}
		r_{1,3}(p)=\frac{1}{3}(\sigma_1(p)-a_p)=\frac{1}{3}(p+1-a_p)=\frac{1}{3} N_{p}(1).
	\end{equation}
	
	We are going to use induction on $k$. The case for $k=1,2$ is evident from \zcref{lem:recursionhelp1,cor:recursionhelp2}.	Now, assume that \zcref{eq:equationforrk} is true for $k=m$, that is,
	\begin{equation}
		3r_{1,3}(p^m)=N(m)+3p\cdot r_{1,3}(p^{m-2}).
	\end{equation}
	Then for $k=m+2$,
	\begin{align}
		3r_{1,3}(p^{m+2})&=\sigma_1(p^{m+2})-a(p^{m+2})\\
		&=N(m+2)+p\sigma_1(p^m)+(\a^{m+2}+\b^{m+2})-a(p^{m+2})\\
		&=N(m+2)+p\sigma_1(p^m)-p\sum_{i=1}^{m}\a^i\b^{m-i}\\
		&=N(m+2)+3p\cdot r_{1,3}(p^m). 
	\end{align}
	
	Using this formula successively, we get the extended form.
	Also, notice that for primes $p\equiv 5\pmod{6}$, $p^{2k+2}\equiv 1 \pmod{6}$. So, the same result will follow. 
\end{proof}

\begin{remark}
	In \cite{CHO2018999}, some examples give rise to linear relations between counts of representation of primes by a congruent quadratic form and point-counting on Elliptic curves over finite fields as shown in \zcref{tab:nonlineartable}. Note that these relations, unlike previous cases, has other terms too.

	\begin{table}[!ht]
		\begin{center}
			\scalebox{0.84}{
				\begin{tabular}{c c c c c c}
					\toprule
					\makecell{Quadratic \\form $Q$}& \makecell{Congruence \\ conditions} & \makecell{Elliptic \\ curve $E$} & \makecell{CM }&\makecell{Relation \\between\\ $N_p^E(1)$ and $r(p)$}  & \makecell{For primes\\ $p$  with \\congruence\\ conditions} \\
					\midrule
					$x_1^2+x_2^2+x_3^2+x_4^2$ &\makecell{ $x_2, x_3, x_4 \equiv 0\pmod{3}$,\\$x_1\equiv 1\pmod{3}$} & $y^2=x^3+1$ & $\Q(\sqrt{-3})$ &\makecell{$3r(p)=3p$\\$+3-2N_p^E(1)$} & \makecell{$p \equiv 1,7$\\
						$\pmod{12}$}\\[15pt]
					
					$x_1^2+x_2^2+2x_3^2+2x_4^2$&\makecell{ $x_2, x_3, x_4 \equiv 0\pmod{2}$,\\$x_1\equiv 1\pmod{2}$} & $y^2=x^3+4x$ & 
					$\Q(\sqrt{-1})$ &\makecell{$r(p)=2p$\\$+2-N_p^E(1)$} & $p \equiv 1\pmod{4}$\\[15pt]
					
					\makecell{$x_1^2+x_1x_2+x_2^2$\\$+x_3^2+x_3x_4+x_4^2$} &\makecell{ $x_2,x_3, x_4 \equiv 0\pmod{3}$,\\$x_1\equiv 1\pmod{3}$} & \makecell{$y^2+y=x^3$ }& {$\Q(\sqrt{-3})$}&\makecell{$3r(p)=3p$\\$+3-2N_p^E(1)$} & $p \equiv 1\pmod{3}$\\[15pt]

					\makecell{$x_1^2+x_1x_2+2x_2^2$\\$+x_3^2+x_3x_4+2x_4^2$} &\makecell{ $x_2,x_3, x_4 \equiv 0\pmod{2}$,\\$x_1\equiv 1\pmod{2}$} & \makecell{$y^2+xy+y$\\ $=x^3-x$}& {None}&\makecell{$3r(p)=6p$\\$+6-4N_p^E(1)$} & $p \equiv 1\pmod{2}$\\ \bottomrule
					
			\end{tabular}}
			\caption{Relations similar to \zcref{thm:r13ecpc}, where $r_Q(p)$ is the number of representations of a prime $p$ by the congruent quadratic form $(Q,(a_i,S_i))$ and $N_p^E(1)$ is $|\{(x,y)\in \F_p\mid (x,y) \text{ is on }E\}|+1$ (including the point at infinity) for the elliptic curve $E$. 
			}
			\label{tab:nonlineartable}
		\end{center}
	\end{table}
\end{remark}

\begin{example}
	For for $k=7$
	\begin{equation}
		a_{p^7}=\a^7+\b^7+pa_p(\a^2+\b^2)^2-p^3a_p.
	\end{equation}
	Hence,
	\begin{equation}
		\frac{1}{3}[N_p(7)+pN_p(5)+p^2N_p(3)+p^3N_p(1)]=\frac{1}{3}[\sigma_1(p^7)-a_{p^7}]=r_{1,3}(p^7).
	\end{equation} 
	Also, for  $k=8$
	\begin{equation}
		a_{p^8}=\a^8+\b^8+pa_p^2(\a^2+\b^2)(\a^2+\b^2-p)-p^3(a_p^2-\a^2-\b^2)+p^4.
	\end{equation}
	Therefore,
	\begin{equation}
		\frac{1}{3}[N_p(8)+pN_p(6)+p^2N_p(4)+p^3N_p(2)]= \frac{1}{3}[\sigma_1(p^8)-a_{p^8}]=r_{1,3}(p^8). 
	\end{equation}
\end{example}

Finally, we aim to prove the general relation in \zcref{thm:r13general}. We already know from the definition and \zcref{prop:propforECPC} that $r_{1,3}(n)=0$ if $n\not  \equiv 1\pmod{3}$. Up next we investigate how to use the prime decomposition of a natural number $n\equiv 1\pmod{6}$ to relate it to the prime counting $N_p(k)$ via  part 3 of \zcref{thm:r13ecpc}. Notice that, for two distinct primes $p,q$ we get that
\begin{equation}
	r_{1,3}(p)=\frac{1}{3}\left[ \sigma_1(p)-a_p\right] ,\hspace*{15pt} r_{1,3}(q)=\frac{1}{3}\left[ \sigma_1(q)-a_q\right].
\end{equation}
Clearly,
\begin{equation}
	r_{1,3}(pq)=\frac{1}{3}\left[ \sigma_1(pq)-a_{pq}\right]=\frac{1}{3}\left[ \sigma_1(p)\sigma_1(q)-a_pa_q\right]\neq r_{1,3}(p)r_{1,3}(q).
\end{equation} Therefore, $r_{1,3}(p)$ is not a multiplicative function.

\begin{proof}[Proof of \zcref{thm:r13general}].

	We will use induction on $k$ the number of distinct primes in the prime decomposition of $n$. The base case when number of distinct primes is obvious. Next, for $k=2$, observe that
	\begin{equation}
		a_{p_1^{\a_1}}r(p_2^{\a_2})+\sigma_1(p_2^{\a_2})r(p_1^{\a_1})=\frac{1}{3}\left[ \sigma_1(p_1^{\a_1}p_2^{\a_2})-a_{p_1^{\a_1}p_2^{\a_2}}\right]=r_{1,3}(p_1^{\a_1}p_2^{\a_2}) 
	\end{equation}
	Next, we assume that that the hypothesis is true for $k=m$, so that we get,
	\begin{equation}\label{eq:induction}
		r_{1,3}\left(\prod_{j=1}^mp_j^{\a_j}\right)=\sigma_1(p_2^{\a_2})\cdots\sigma_1(p_m^{\a_m})r_{1,3}(p_1^{\a_1})+\cdots+a_{p_1^{\a_1}}\dots a_{p_{m-1}^{\a_{m-1}}}r_{1,3}(p_m^{\a_m}). 
	\end{equation}
	Therefore, using \zcref{eq:induction} we get for $k=m+1$;
	\begin{align}
		&r_{1,3}\left(\prod_{i=1}^mp_i^{\a_i}\cdot p_{m+1 }^{\a_{m+1}}\right)\\
		&=\sigma_1(p_2^{\a_2})\cdots\sigma_1(p_m^{\a_m}p_{m+1}^{\a_{m+1}})r_{1,3}(p_1^{\a_1})+\cdots+a_{p_1^{\a_1}}\dots a_{p_{m-1}^{\a_{m-1}}}r_{1,3}(p_m^{\a_m}p_{m+1 }^{\a_{m+1}})\\
		&=\begin{multlined}[t]
			\sigma_1(p_2^{\a_2})\cdots\sigma_1(p_m^{\a_m})\sigma_1(p_{m+1 }^{\a_{m+1}})r_{1,3}(p_1^{\a_1})+\cdots \\
			\hspace{70pt} +a_{p_1^{\a_1}}\dots a_{p_{m-1}^{\a_{m-1}}}[\sigma_1(p_{m+1}^{\a_{m+1}})r_{1,3}(p_m^{\a_m})+a_{p_m^{\a_m}}r_{1,3}(p_{m+1}^{\a_{m+1}})]
		\end{multlined}\\
		&=\sum_{i=1}^{m+1}\left(\prod_{j=1}^{i-1}a_{p_j^{\a_j}}\cdot\prod_{l=i+1}^{m+1}\sigma_1(p_l^{\a_l})\right)r_{1,3}(p_i^{\a_i}). \qedhere
	\end{align}
\end{proof}
	
\subsection{Application}
In this subsection, we show a  consequence of \zcref{thm:r13ecpc}. 
\subsubsection{\texorpdfstring{$N_p(k)$}{N\_p (k)} and Gr\"{o}ssencharacter of higher exponents }\label{subsec:chsec}

In order to discuss the higher exponents of the Gr$\ddot{o}$ssencharacter $\chi$ associated with the CM cusp form $\eta(6\tau)^4$, first we need to define the character precisely. We can get an explicit description of the character $\chi$ of the id\`ele class group  of $\mathbb{Q}(\sqrt{-3})$ and of algebraic of type 2 such that
\begin{equation}
	L(\eta(6\tau)^4,s)=L(\chi,s-\frac{1}{2}).
\end{equation}

The imaginary quadratic field $\mathbb{Q}(\sqrt{-3})=\mathbb{Q}(\zeta_6)$ has class number 1 and discriminant -3, so 3 is the only ramified prime in $\mathbb{Q}(\zeta_6)$ and primes $p\equiv 2\pmod{3}$ are the inert ones. Let $S$ be the set of non-zero integral ideals of $\mathbb{Z}[\zeta_6]$ which are coprime to 6. Recall that, that every ideal of $S$ has a unique generator of the form $a+b\sqrt{-3}, a,b\in \mathbb{Z},a\equiv1\pmod{3},a\not\equiv b\pmod{2}$. [See  \cite{MR3848417}, Section 4 for more details]. The character is defined on id\`eles as,
\begin{equation}
	\tilde{\chi}(\mathcal{J})\coloneqq a-b\sqrt{-3}
\end{equation}
where $\mathcal{J}=(a+b\sqrt{-3})\in S$.

Let $v_2,v_3$ are the only places above 2 and 3 respectively, with the norm $Nv_2=4,Nv_3=3.$ Denote for $k=1,2$, $M_k$ be the maximal ideal at $v_j,$ and $U_k$ be the group of units at $v_k$. Denote $\chi_{v_k}$ the restriction of $\chi$ at the place $v_k$ of $\mathbb{Q}(\zeta_6)$.

Notice that,  the algebraic type 2 condition means that at the complex place $v_{\infty}$, $\chi_{v_{\infty}}(z)=\frac{z}{|z|}$ for all $z\in \mathbb{C}^{\times}$.

Moreover, $\chi_{v_3}$ is non-trivial on $U_3/(1+M_3)=\langle-1\rangle$ therefore, $\chi_{v_3}(-1)=-1$. On the other hand, $\zeta_3$ is a unit of $\mathbb{Q}(\zeta_6)$ of order 3, therefore, $\chi_{v_3}(\zeta_3)=1$.

Also, $U_2/(1+M_2)=\langle\zeta_3\rangle$ is a cyclic group of order 3. Therefore, we get that $\chi(\zeta_3)=1=\prod_{v}\chi_{v}(\zeta_3)=\chi_{v_2}(\zeta_3)\chi_{v_{\infty}}(\zeta_3)=\zeta_3\chi_{v_2}(\zeta_3)$. Therefore, we get $\chi_{v_2}(\zeta_3)=\zeta_3^{-1}=\zeta_3^{2}$.

Lastly, we look at the relation between the level of the modular form and the norm of the conductor of the corresponding Gr\"{o}ssencharacter. From \cite{book:silverman}, we get that for an elliptic curve $E/L$, with corresponding Gr\"{o}ssencharacter $\chi$ of conductor $c$, the completed $L$-function,
\begin{equation}
	\Lambda(E/L,s)=(N_{L/\mathbb{Q}}(D \cdot c))^{s/2}((2\pi)\Gamma(s))^{[L:\mathbb{Q}]}L(E/L,s),
\end{equation}
where $D$ is the different of $L/\mathbb{Q}$. From \cite{lmfdb}, we get that the completed $L$-function for the corresponding modular form $f$ of level $N$,

\begin{equation}
	\Lambda(f,s)=N^{s/2}\prod_{j=1}^J\Gamma_{\mathbb{R}}(s+\mu_j)\prod_{k=1}^K\Gamma_{\mathbb{C}}(s+v_k)L(f,s).
\end{equation}
As $\Lambda(f,s)=\Lambda(E/L,s)$, we get
\begin{equation}\label{eq:leveleq}
	N=N_{L/\mathbb{Q}}(D \cdot c)=|\Delta_L|N_{L/\mathbb{Q}}(c),
\end{equation} 
where $N$ is the level of $f$ and $\Delta_L$ is the discriminant of the field $L$. 
Now, $2$ splits in $\mathbb{Q}(\zeta_6)/\mathbb{Q}$, $N_{\mathbb{Q}(\zeta_6)/\mathbb{Q}}(2)=4$ while $3$ is ramified so there exists $\rho$ such that $3=\rho.\Bar{\rho}$ and $N_{\mathbb{Q}(\zeta_6)/\mathbb{Q}}(\rho)=3$.
We know that for the Gr$\ddot{o}$ssecharater $\chi=\chi_f\chi_{\infty}$, and the norm of the conductor of $\chi$ is actually the conductor of the Dirichlet character $\chi_f$ which has contributions from the conductors of both $\chi_{v_2}$ and $\chi_{v_3}$. For $\chi$ in our case, the conductor will be of the form $c=2\rho$ giving us $N_{L/\mathbb{Q}}(c)=N_{L/\mathbb{Q}}(2)N_{L/\mathbb{Q}}(\rho)=4\cdot 3=12$ which also matches \zcref{eq:leveleq}. The discriminant $\Delta_L=-3$ in our case.

Now, for any positive power $k$, let $f_k$ be the corresponding modular form of $\chi^k$. Then $f_k$ will have weight $k+1$. For the level, we will use \zcref{eq:leveleq}. We will get that when $k\equiv 0,2,4\pmod{6}$, $\chi_{v_3}^k$ vanishes while $\chi_{v_2}^k$ vanishes when $k\equiv 0,3\pmod{6}$. So we get
\begin{equation}
	N_{L/\mathbb{Q}}(c_{\chi^k})=\begin{cases}
		12, & \text{if } k\equiv 1,5\pmod{6}, \\
		4,  & \text{if } k\equiv 2,4\pmod{6}, \\
		3,  & \text{if } k\equiv 3\pmod{6},   \\
		1,  & \text{if } k\equiv 0\pmod{6},
	\end{cases}
\end{equation}
where $c_{\chi^k}$ is the conductor of the character $\chi^k$.
Putting this in \zcref{eq:leveleq} we can easily obtain the level of the modular form $f_{\chi^k}$ corresponding to the character $\chi^k$. Finally, we can get the actual Fourier expansion of $f_{\chi^k}$.

We denote $a_k(p)$ as the $a_p$ coefficients of $L(\chi^k,s)$. Then note that $a_k(p)=\a^k+\b^k$ where $\a,\b$ are as given in \zcref{thm:weil}. We can get the following three-term recursion relation of $a_k(p)$ similar to \zcref{thm:recursionap}. 
\begin{proposition}
	For any prime $p$,
	\begin{equation}
		a_k(p)=a_{k-1}(p)a_1(p)-pa_{k-2}(p).
	\end{equation}
\end{proposition}
\begin{proof}
	Notice that, for $\a$ and $\b$ defined in \zcref{thm:weil}, we have
	\begin{equation}
		\a^k+\b^k=(\a^{k-1}+\b^{k-1})(\a+\b)-\a\b(\a^{k-2}+\b^{k-2}).
	\end{equation}
	The proof follows from realizing the terms in terms of $a_r(p)$. 
\end{proof}

\begin{remark}
	Notice that, \zcref{thm:weil} gives us,
	\begin{equation}\label{eq:nkak}
		N_p(k)=p^k+1-a_k(p).
	\end{equation}
	Hence, using \zcref{eq:nkak} and \zcref{thm:r13ecpc} we can relate $r_{1,3}(p^k)$ with coefficients of $L(\chi^k,s)$. 
\end{remark}

\section*{Declarations}


\begin{itemize}
\item Funding: The author is partially supported by NSF grants DMS-2302514 and DMS-2302531 and the 2024 summer research assistantship awarded by the Department of Mathematics at Louisiana State University.
\item Conflict of interest/Competing interests: Not applicable.
\item Ethics approval and consent to participate: Not applicable.
\item Consent for publication: Not applicable.
\item Data availability: Not applicable.
\item Materials availability: Not applicable.
\item Code availability: Not applicable
\item Author contribution: Not applicable.
\end{itemize}







\begin{appendices}

\section{Basis of \texorpdfstring{$\mathcal{E}_2(\Gamma_0(2^k))$}{the space of weight 2 Eisenstein series of level 2\string^k} for \texorpdfstring{$k\leq 7$}{k < 8}}\label{app:evenbasis}

We give a basis of $\mathcal{E}_2(\Gamma_0(2^n))$ for $n\leq 7$. Here we use $c=-\frac{1}{24}$ normalization. We follow the method described in the proof of \zcref{thm:basisgenerationthm}. In that process, we omit some elements as they are just the linear combination of previous ones. For example, when we consider $\widehat{E}_2|S_{4,1}, \widehat{E}_2|S_{4,3}$, we omit $\widehat{E}_2|S_{2,1}=\widehat{E}_2|S_{4,1}+\widehat{E}_2|S_{4,3}$. We start of with the elements and their fourier expansions.

\begin{newthm}{Basis of $\mathcal{E}_2(\Gamma_0(2))$}
	By \zcref{eq:exactdimension}, $\dim (\mathcal{E}_2(\Gamma_0(2)))=1$. The following is a basis set for the Eisenstein space, 
	\begin{equation}
		\mathbb{E}_{2,1}(\tau) \coloneqq \widehat{E}_2|(2V_2-S_{2,0})|V_{\frac{1}{2}}(\tau)=1+24q+24q^{2}+96q^{3}+24q^{4}+144q^{5}+\cdots
	\end{equation}
    Recall that, $E_{2,2}(\tau)=2\widehat{E}_2(2\tau)-\widehat{E}_2(\tau)$ is the widely used basis for $\mathcal{E}_2(\G_0(2))$ and it satisfies, 
    \begin{align}
        E_{2,2}(\tau)=\mathbb{E}_{2,1}(\tau).
    \end{align}
    
\end{newthm}

\begin{newthm}{Basis of $\mathcal{E}_2(\Gamma_0(2^2))$}
	By \zcref{eq:exactdimension}, $\dim (\mathcal{E}_2(\Gamma_0(2^2)))=2$. The following is a basis set for the Eisenstein space,
	\begin{align}
		\EE{4}{1}(\tau) &\coloneqq \widehat{E}_2|(2V_2-S_{2,0})(\tau)=1+24q^2+24q^{4}+96q^{6}+24q^{8}+144q^{10}+ \cdots\\
		\EE{4}{2}(\tau) &\coloneqq c\widehat{E}_2|S_{2,1}(\tau)=q+4q^3+6q^5+8q^7+13q^9+12q^{11}+\cdots
	\end{align}
	Recall, the Jacobian theta series $\theta_3(2\tau)^4$ and  $\theta_4(2\tau)^4$ from \zcref{eq:jacobitheta}, can be decomposed as
	\begin{equation}
		\theta_3(2\tau)^4= \EE{4}{1}(\tau)+8\EE{4}{2}(\tau), \quad  \theta_4(2\tau)^4= \EE{4}{1}(\tau)-8\EE{4}{2}(\tau).
	\end{equation}		
\end{newthm}	

\begin{newthm}{Basis of $\mathcal{E}_2(\Gamma_0(2^3))$}
	By \zcref{eq:exactdimension}, $\dim (\mathcal{E}_2(\Gamma_0(2^3)))=3$. A basis set for the Eisenstein space is given by
	\begin{align}
		\EE{8}{1}(\tau)\coloneqq& \widehat{E}_2|(2V_2-S_{2,0})|V_2(\tau)=1+24q^4+24q^{8}+96q^{12}+24q^{16}+144q^{20}+\cdots\\
		\EE{8}{2}(\tau)\coloneqq& c\widehat{E}_2|S_{2,1}(\tau)=q+4q^3+6q^5+8q^7+13q^9+12q^{11}+\cdots\\
		\EE{8}{3}(\tau)\coloneqq& c\widehat{E}_2|S_{2,1}|V_2(\tau)=q^2+4q^6+6q^{10}+8q^{14}+13q^{18}+12q^{22}+\cdots
	\end{align}		
\end{newthm}
	
\begin{newthm}{Basis of $\mathcal{E}_2(\Gamma_0(2^4))$}
	By \zcref{eq:exactdimension}, $\dim (\mathcal{E}_2(\Gamma_0(2^4)))=5$.  So, a normalized basis for the Eisenstein space is as follows,
	\begin{align}
		\EE{16}{1}(\tau)\coloneqq& \widehat{E}_2|(2V_2-S_{2,0})|V_4(\tau)=1+24q^8+24q^{16}+96q^{24}+24q^{32}+\cdots\\
		\EE{16}{2}(\tau)\coloneqq& c\widehat{E}_2|S_{4,1}(\tau)=q+6q^5+13q^9+14q^{13}+18q^{17}+32q^{21}+\cdots\\
		\EE{16}{3}(\tau)\coloneqq& \frac{c}{4} \widehat{E}_2|S_{4,3}(\tau)=q^3+2q^7+3q^{11}+6q^{15}+5q^{19}+6q^{23}+\cdots\\
		\EE{16}{4}(\tau)\coloneqq&c\widehat{E}_2|S_{2,1}|V_2(\tau)=q^2+4q^6+6q^{10}+8q^{14}+13q^{18}+12q^{22}+\cdots\\
		\EE{16}{5}(\tau)\coloneqq& c\widehat{E}_2|S_{2,1}|V_4(\tau)=q^{4}+4q^{12}+6q^{20}+8q^{28}+\cdots
	\end{align}
\end{newthm}

\begin{newthm}{Basis of $\mathcal{E}_2(\Gamma_0(2^5))$}
	By \zcref{eq:exactdimension}, $\dim (\mathcal{E}_2(\Gamma_0(2^5)))=7$. The following is a basis for Eisenstein space, 
	\begin{align}
		\EE{32}{1}(\tau)\coloneqq& \widehat{E}_2|(2V_2-S_{2,0})|V_8(\tau)=1+24q^{16}+24q^{32}+96q^{48}+\cdots\\
		\EE{32}{2}(\tau)\coloneqq& c\widehat{E}_2|S_{4,1}(\tau)=q+6q^5+13q^9+14q^{13}+18q^{17}+32q^{21}+\cdots\\
		\EE{32}{3}(\tau)\coloneqq& \frac{c}{4}\widehat{E}_2|S_{4,3}(\tau)=q^3+2q^7+3q^{11}+6q^{15}+5q^{19}+6q^{23}+\cdots\\
		\EE{32}{4}(\tau)\coloneqq& c\widehat{E}_2|S_{4,1}|V_2(\tau)=q^2+6q^{10}+13q^{18}+14q^{26}+18q^{34}+32q^{42}+\cdots\\
		\EE{32}{5}(\tau)\coloneqq& \frac{c}{4}\widehat{E}_2|S_{4,3}|V_2(\tau)=q^6+2q^{14}+3q^{22}+6q^{30}+5q^{38}+6q^{46}+\cdots\\
		\EE{32}{6}(\tau)\coloneqq& c\widehat{E}_2|S_{2,1}|V_4(\tau)=q^4+4q^{12}+6q^{20}+8q^{28}+\cdots\\
		\EE{32}{7}(\tau)\coloneqq& c\widehat{E}_2|S_{2,1}|V_8(\tau)=q^{8}+4q^{24}+6q^{40}+8q^{56}+\cdots
	\end{align}
\end{newthm}

\begin{newthm}{Basis of $\mathcal{E}_2(\Gamma_0(2^6))$}
	By \zcref{eq:exactdimension}, $\dim (\mathcal{E}_2(\Gamma_0(2^6)))=11$. A basis set for the Eisenstein space is given by
	\begin{align}
		\EE{64}{1}(\tau)\coloneqq& \widehat{E}_2|(2V_2-S_{2,0})|V_{16}(\tau)=1+24q^{32}+24q^{64}+96q^{96}+\cdots\\
		\EE{64}{2}(\tau)\coloneqq& c\widehat{E}_2|S_{8,1}(\tau)=q+13q^9+18q^{17}+31q^{25}+\cdots\\
		\EE{64}{3}(\tau)\coloneqq& \frac{c}{4}\widehat{E}_2|S_{8,3}(\tau)=q^3+3q^{11}+5q^{19}
		+10q^{27}+\cdots\\
		\EE{64}{4}(\tau)\coloneqq& \frac{c}{2}\widehat{E}_2|S_{8,5}(\tau)=3q^5+7q^{13}+16q^{21}+15q^{29}+19q^{37}+\cdots\\
		\EE{64}{5}(\tau)\coloneqq& \frac{c}{8}\widehat{E}_2|S_{8,7}(\tau)=q^7+3q^{15}+3q^{23}+4q^{31}+\cdots\\
		\EE{64}{6}(\tau)\coloneqq& \frac{c}{3}\widehat{E}_2|S_{4,1}|V_2(\tau)=q^2+6q^{10}+13q^{18}+14q^{26}+\cdots\\
		\EE{64}{7}(\tau)\coloneqq& \frac{c}{4}\widehat{E}_2|S_{4,3}|V_2(\tau)=q^6+2q^{14}+3q^{22}+6q^{30}+\cdots\\
		\EE{64}{8}(\tau)\coloneqq& \frac{c}{3}\widehat{E}_2|S_{4,1}|V_4(\tau)=q^4+6q^{20}+13q^{26}+14q^{52}+\cdots\\
		\EE{64}{9}(\tau)\coloneqq& \frac{c}{4}\widehat{E}_2|S_{4,3}|V_4(\tau)=q^{12}+2q^{28}+3q^{44}+6q^{60}+5q^{76}+\cdots\\
		\EE{64}{10}(\tau)\coloneqq& c\widehat{E}_2|S_{2,1}|V_8(\tau)=q^8+4q^{24}+6q^{40}+8q^{56}+\cdots\\
		\EE{64}{11}(\tau)\coloneqq& c\widehat{E}_2|S_{2,1}|V_{16}(\tau)=q^{16}+4q^{48}+6q^{80}+\cdots
	\end{align}
\end{newthm}

\begin{newthm}{Basis of $\mathcal{E}_2(\Gamma_0(2^7))$}		
	By \zcref{eq:exactdimension}, $\dim (\mathcal{E}_2(\Gamma_0(2^7)))=15$. The following is a basis set for the Eisenstein space, 
	\begin{align}
		\EE{128}{1}(\tau)\coloneqq& \widehat{E}_2|(2V_2-S_{2,0})|V_{32}(\tau)=1+24q^{6
		}+24q^{128}+96q^{192}+\cdots\\
		\EE{128}{2}(\tau)\coloneqq& c\widehat{E}_2|S_{8,1}(\tau)=q+13q^9+18q^{17}+31q^{25}+\cdots\\
		\EE{128}{3}(\tau)\coloneqq& \frac{c}{4}\widehat{E}_2|S_{8,3}(\tau)=q^3+3q^{11}+5q^{19}
		+10q^{27}+\cdots\\
		\EE{128}{4}(\tau)\coloneqq& \frac{c}{2}\widehat{E}_2|S_{8,5}(\tau)=3q^5+7q^{13}+16q^{21}+15q^{29}+19q^{37}+\cdots\\
		\EE{128}{5}(\tau)\coloneqq& \frac{c}{8}\widehat{E}_2|S_{8,7}(\tau)=q^7+3q^{15}+3q^{23}+4q^{31}+\cdots\\
		\EE{128}{6}(\tau)\coloneqq& c\widehat{E}_2|S_{8,1}|V_2(\tau)=q^2+13q^{18}+18q^{34}+31q^{50}+\cdots\\
		\EE{128}{7}(\tau)\coloneqq& \frac{c}{4}\widehat{E}_2|S_{8,3}|V_2(\tau)=q^6+3q^{22}+5q^{38}
		+10q^{54}+\cdots\\
		\EE{128}{8}(\tau)\coloneqq& \frac{c}{2}\widehat{E}_2|S_{8,5}|V_2(\tau)=3q^{10}+7q^{26}+16q^{42}+15q^{58}+19q^{74}+\cdots\\
		\EE{128}{9}(\tau)\coloneqq& \frac{c}{8}\widehat{E}_2|S_{8,7}|V_2(\tau)=q^{14}+3q^{30}+3q^{46}+4q^{62}+\cdots\\
		\EE{128}{10}(\tau)\coloneqq& \frac{c}{3}\widehat{E}_2|S_{4,1}|V_4(\tau)=q^4+6q^{20}+13q^{36}+14q^{52}+\cdots\\
		\EE{128}{11}(\tau)\coloneqq& \frac{c}{4}\widehat{E}_2|S_{4,3}|V_4(\tau)=q^{12}+2q^{28}+3q^{44}+6q^{60}+5q^{76}+\cdots\\
		\EE{128}{12}(\tau)\coloneqq& \frac{c}{3}\widehat{E}_2|S_{4,1}|V_8(\tau)=q^8+6q^{40}+13q^{72}+14q^{104}+\cdots\\
		\EE{128}{13}(\tau)\coloneqq& \frac{c}{4}\widehat{E}_2|S_{4,3}|V_8(\tau)=q^{24}+2q^{56}+3q^{88}+6q^{128}+5q^{152}+\cdots\\
		\EE{128}{14}(\tau)\coloneqq& c\widehat{E}_2|S_{2,1}|V_{16}(\tau)=q^{16}+4q^{48}+6q^{80}+\cdots\\
		\EE{128}{15}(\tau)\coloneqq& c\widehat{E}_2|S_{2,1}|V_{32}(\tau)=q^{32}+4q^{96}+6q^{160}+\cdots
	\end{align}
\end{newthm}

\section{Basis for \texorpdfstring{$\mathcal{E}_2(\Gamma_0(p^k))$}{the space of weight 2 Eisenstein series of level p\string^k} for odd prime \texorpdfstring{$p$}{p}}\label{app:oddbasis}

\begin{newthm}{Basis of $\mathcal{E}_2(\Gamma_0(p))$}
	By \zcref{eq:exactdimension}, $\dim (\mathcal{E}_2(\Gamma_0(p)))=1$. So we get the following element generating the whole Eisenstein space,
	\begin{equation}
		\EE p 1(\tau)\coloneqq \frac{1}{p-1}\widehat{E}_2|(pV_p-S_{p,0})|V_{\frac{1}{p}}(\tau)
	\end{equation}
\end{newthm}

\begin{newthm}{Basis of $\mathcal{E}_2(\Gamma_0(3^2))$}
	Followed by the discussion in \zcref{sec:basissection}, we can only get a complete set of basis of $\mathcal{E}_2(\Gamma_0(p^k))$ for $p=3$ when $k>1$. By \zcref{eq:exactdimension}, $\dim (\mathcal{E}_2(\Gamma_0(3^2)))=3$. So, a normalized basis for the Eisenstein space is as follows,
	\begin{align}
		\EE 9 1(\tau)\coloneqq& \frac{1}{2}\widehat{E}_2|(3V_3-S_{3,0})(\tau)\\
		\EE 9 2(\tau)\coloneqq& c \cdot \widehat{E}_2|S_{3,1}(\tau)\\
		\EE 9 3(\tau)\coloneqq&\frac{c}{3} \cdot\widehat{E}_2|S_{3,2}(\tau)
	\end{align}
\end{newthm}

\begin{newthm}{Basis of $\mathcal{E}_2(\Gamma_0(3^3))$}
	By \zcref{eq:exactdimension}, $\dim (\mathcal{E}_2(\Gamma_0(3^3)))=5$. So, a normalized basis for the Eisenstein space is as follows,
	\begin{align}
		\EE{27}{1}(\tau)\coloneqq& \frac{1}{2}\widehat{E}_2|(3V_3-S_{3,0})|V_3(\tau)\\
		\EE{27}{2}(\tau)\coloneqq& c \cdot \widehat{E}_2|S_{3,1}(\tau)\\
		\EE{27}{3}(\tau)\coloneqq& \frac{c}{3} \widehat{E}_2|S_{3,2}(\tau)\\
		\EE{27}{4}(\tau)\coloneqq& c \cdot \widehat{E}_2|S_{3,1}|V_3(\tau)\\
		\EE{27}{5}(\tau)\coloneqq& \frac{c}{3} \cdot \widehat{E}_2|S_{3,2}|V_3(\tau)
	\end{align}
\end{newthm}

\begin{remark*}
	For spaces $\mathcal{E}_2(\Gamma_0(3^k)), k>3$, we need to incorporate elements $\widehat{E}_2|S_{9,m}$ but as noted before, they will be no longer linearly independent. Therefore, we cannot get the whole basis just by using $S,V$ operators. 
\end{remark*}

\section{Table of cusp values for \texorpdfstring{$\Gamma_0(16)$}{Gamma\_0 (16)}}\label{app:gamma016table}

\begin{table}[!ht]
	\renewcommand{\arraystretch}{1.2}
	\begin{center}
		\scalebox{1.1}{
			\begin{tabular}{c c c c c c c}
				\toprule
				Cusps                &  $\frac{1}{8}$  &  $\frac{1}{4}$   &  $\frac{3}{4}$   &  $\frac{1}{2}$   &         0         & $\infty$ \\ \midrule
				values of $\EE{16}{1}(\tau)$ & $\frac{53}{48}$ & $\frac{53}{192}$ & $\frac{53}{192}$ & $\frac{53}{768}$ & $\frac{53}{3072}$ &    1     \\[5pt]
				values of $\EE{16}{2}(\tau)$ & $-\frac{1}{2}$  &  $-\frac{i}{8}$  &  $-\frac{i}{8}$  &  $\frac{1}{32}$  & $-\frac{1}{128}$  &    0     \\[5pt]
				values of $\EE{16}{3}(\tau)$ & $-\frac{1}{8}$  &  $\frac{i}{32}$  &  $\frac{i}{32}$  & $\frac{1}{128}$  & $-\frac{1}{512}$  &    0     \\[5pt]
				values of $\EE{16}{4}(\tau)$ & $-\frac{1}{4}$  &  $\frac{1}{16}$  &  $\frac{1}{16}$  & $-\frac{1}{64}$  & $-\frac{1}{256}$  &    0     \\[5pt]
				values of $\EE{16}{5}(\tau)$ & $\frac{1}{16}$  & $-\frac{1}{64}$  & $-\frac{1}{64}$  & $-\frac{1}{256}$ & $-\frac{1}{1024}$ &    0     \\ \bottomrule
			\end{tabular}
		}
		\caption{Values of basis elements of $\mathcal{E}_2(\G_0(16))$ at cusps.}
	\end{center}
\end{table}

\newpage
\section{Table of cusp values for \texorpdfstring{$\Gamma_0(36)$}{Gamma\_0 (36)}}\label{app:gamma036table}

\begin{table}[!ht]
	\renewcommand{\arraystretch}{1.2}
	\begin{center}
		\scalebox{0.875}{
			\begin{tabular}{c c c c c c c c c c c c c}
				\toprule
			      Cusps        & 0                   & $\frac{1}{18}$   & $\frac{1}{12}$           & $\frac{1}{9}$     & $\frac{1}{6}$             & $\frac{1}{4}$     & $\frac{1}{3}$             & $\frac{5}{12}$           & $\frac{1}{2}$      & $\frac{2}{3}$             & $\frac{5}{6}$             & $\infty$ \\ \midrule
			 values of $B_1$   & $\frac{101}{15552}$ & $\frac{101}{48}$ & $\frac{49}{216}$         & $\frac{101}{192}$ & $\frac{101}{432}$         & $\frac{49}{1944}$ & $\frac{101}{3888}$        & $\frac{49}{216}$         & $\frac{101}{3888}$ & $\frac{101}{1728}$        & $\frac{101}{432}$         & 1        \\[5pt]
			 values of $B_2$   & $-\frac{1}{216}$    & $0$              & 0                        & $0$               & $-\frac{1+i\sqrt{3}}{12}$ & $0$               & $\frac{1-i\sqrt{3}}{48}$  & $0$                      & $\frac{1}{54}$     & $\frac{1+i\sqrt{3}}{48}$  & -$\frac{1-i\sqrt{3}}{12}$ & 0        \\[5pt]
			 values of $B_3$   & $-\frac{1}{216}$    & $0$              & 0                        & $0$               & $-\frac{1-i\sqrt{3}}{12}$ & $0$               & $\frac{1+i\sqrt{3}}{48}$  & $0$                      & $\frac{1}{54}$     & $\frac{1-i\sqrt{3}}{48}$  & -$\frac{1+i\sqrt{3}}{12}$ & 0        \\[5pt]
			 values of $B_4$   & $-\frac{1}{81}$     & $0$              & $\frac{1-i\sqrt{3}}{18}$ & $0$               & $\frac{1+i\sqrt{3}}{18}$  & $-\frac{1}{81}$   & $\frac{1-i\sqrt{3}}{18}$  & $\frac{1+i\sqrt{3}}{18}$ & $-\frac{1}{81}$    & $\frac{1+i\sqrt{3}}{18}$  & $\frac{1-i\sqrt{3}}{18}$  & 0        \\[5pt]
			 values of $B_5$   & $-\frac{1}{324}$    & $0$              & $\frac{1+i\sqrt{3}}{18}$ & $0$               & $\frac{1-i\sqrt{3}}{18}$  & $-\frac{1}{81}$   & $\frac{1+i\sqrt{3}}{72}$  & $\frac{1-i\sqrt{3}}{18}$ & $-\frac{1}{81}$    & $\frac{1-i\sqrt{3}}{72}$  & $\frac{1+i\sqrt{3}}{18}$  & 0        \\[5pt]
			 values of $B_6$   & $-\frac{1}{1296}$   & $0$              & $\frac{1-i\sqrt{3}}{18}$ & $0$               & $\frac{1+i\sqrt{3}}{72}$  & $-\frac{1}{81}$   & $\frac{1-i\sqrt{3}}{288}$ & $\frac{1+i\sqrt{3}}{18}$ & $-\frac{1}{324}$   & $\frac{1+i\sqrt{3}}{288}$ & $\frac{1-i\sqrt{3}}{72}$  & 0        \\[5pt]
			 values of $B_7$   & $-\frac{1}{81}$     & $0$              & $\frac{1+i\sqrt{3}}{18}$ & $0$               & $\frac{1-i\sqrt{3}}{18}$  & $-\frac{1}{81}$   & $\frac{1+i\sqrt{3}}{18}$  & $\frac{1-i\sqrt{3}}{18}$ & $-\frac{1}{81}$    & $\frac{1-i\sqrt{3}}{18}$  & $\frac{1+i\sqrt{3}}{18}$  & 0        \\[5pt]
			 values of $B_8$   & $-\frac{1}{324}$    & $0$              & $\frac{1-i\sqrt{3}}{18}$ & $0$               & $\frac{1+i\sqrt{3}}{18}$  & $-\frac{1}{81}$   & $\frac{1-i\sqrt{3}}{72}$  & $\frac{1+i\sqrt{3}}{18}$ & $-\frac{1}{81}$    & $\frac{1+i\sqrt{3}}{72}$  & $\frac{1-i\sqrt{3}}{18}$  & 0        \\[5pt]
			 values of $B_9$   & $-\frac{1}{1296}$   & $0$              & $\frac{1+i\sqrt{3}}{18}$ & $0$               & $\frac{1-i\sqrt{3}}{72}$  & $-\frac{1}{81}$   & $\frac{1+i\sqrt{3}}{288}$ & $\frac{1-i\sqrt{3}}{18}$ & $-\frac{1}{324}$   & $\frac{1-i\sqrt{3}}{288}$ & $\frac{1+i\sqrt{3}}{72}$  & 0        \\[5pt]
				values of $B_{10}$ & $-\frac{1}{64}$     & $\frac{1}{16}$   & 0                        & $-\frac{1}{64}$   & $\frac{1}{16}$            & 0                 & $-\frac{1}{64}$           & 0                        & $\frac{1}{16}$     & $-\frac{1}{64}$           & $\frac{1}{16}$            & 0        \\[5pt]
				values of $B_{11}$ & $-\frac{1}{576}$    & $\frac{1}{16}$   & 0                        & $-\frac{1}{64}$   & $\frac{1}{16}$            & 0                 & $-\frac{1}{64}$           & 0                        & $\frac{1}{144}$    & $-\frac{1}{64}$           & $\frac{1}{16}$            & 0\\ \bottomrule
			\end{tabular}
		}
		\caption{Values of basis elements of $\mathcal{E}_2(\G_0(36))$ at cusps.}
	\end{center}
\end{table}




\end{appendices}


\bibliography{bib}

\end{document}